\documentclass[10pt]{amsart}
  \baselineskip 1 mm

\usepackage{amsmath,amsthm,amsfonts,latexsym,amsopn,verbatim,amscd,amssymb}
\usepackage{hyperref}

\theoremstyle{plain}
\newtheorem{theorem}{Theorem}[section]

\newtheorem{corollary}[theorem]{Corollary}
\theoremstyle{definition}

\newtheorem{remark}[theorem]{Remark}

\newtheorem{conjecture}[theorem]{Conjecture}
\newtheorem{ques}[theorem]{Question}

\numberwithin{equation}{section}

\DeclareMathOperator{\spec}{Spec}
\DeclareMathOperator{\L-spec}{L-spec}
\DeclareMathOperator{\Q-spec}{Q-spec}

\newcommand{\bnum}{\begin{enumerate}}
\newcommand{\enum}{\end{enumerate}}

\begin{document}

\title[Spectral aspects of commuting conjugacy class graph of finite  groups]{Spectral aspects of commuting conjugacy class graph of finite  groups}
\author[P. Bhowal and R. K. Nath]{Parthajit Bhowal  and Rajat Kanti Nath$^*$}
\address{Department of Mathematical Sciences, Tezpur
University,  Napaam-784028, Sonitpur, Assam, India.
}
\email{bhowal.parthajit8@gmail.com, rajatkantinath@yahoo.com (correcponding author)}


%

\thanks{*Corresponding author}

%
%
\begin{abstract}
The commuting conjugacy class graph of a non-abelian group $G$, denoted by $\mathcal{CCC}(G)$, is a simple undirected  graph whose vertex set is the set of conjugacy classes of the non-central elements of $G$ and two distinct vertices $x^G$ and $y^G$ are adjacent if there exists some elements $x' \in x^G$ and $y' \in y^G$ such that $x'y' = y'x'$. In this paper we compute various spectra and energies of commuting conjugacy class graph of the groups $D_{2n}, Q_{4m}, U_{(n, m)}, V_{8n}$ and  $SD_{8n}$. Our computation shows that $\mathcal{CCC}(G)$ is super integral for these groups. We compare various energies and as a consequence it is observed that $\mathcal{CCC}(G)$ satisfy E-LE Conjecture of Gutman et al. We also provide negative answer to a question posed by Dutta et al. comparing Laplacian and Signless Laplacian energy. Finally, we conclude this paper by characterizing the above mentioned groups $G$ such that $\mathcal{CCC}(G)$ is hyperenergetic, L-hyperenergetic or Q-hyperenergetic. 
\end{abstract}

\maketitle

%
%

\section{Introduction} \label{S:intro}

 The commuting graph of a finite non-abelian group $G$ with center $Z(G)$ is a simple undirected  graph whose vertex set is $G\setminus Z(G)$ and two distinct vertices $x$ and $y$ are adjacent if  $xy = yx$. This graph was first considered by Brauer and Fowler \cite{bF1955}. Later on, many mathematicians have extended this graph by considering nilpotent graphs, solvable graphs, commuting conjugacy class graphs etc. The commuting conjugacy class graph of a non-abelian group $G$, denoted by $\mathcal{CCC}(G)$, is a simple undirected  graph whose vertex set is the set of conjugacy classes of the non-central elements of $G$ and two distinct vertices $x^G$ and $y^G$ are adjacent if there exists some elements $x' \in x^G$ and $y' \in y^G$ such that $x'y' = y'x'$.  
  The notion of commuting conjugacy class graph of groups was introduced by Herzog, Longobardi and Maj \cite{hLM2009} in the year 2009. However, in their definition of $\mathcal{CCC}(G)$, the vertex set is considered to be the set of all non-identity conjugacy classes of $G$. In the year 2016,  Mohammadian et al. \cite{mefw} have classified all finite groups such that $\mathcal{CCC}(G)$ is triangle-free. Recently, in \cite{sA2020}, Salahshour and Ashrafi  have obtain the structure of $\mathcal{CCC}(G)$ considering $G$ to be the following groups: 
\begin{align*}
D_{2n} &= \langle x, y : x^n = y^2 = 1, yxy = x^{-1}\rangle \text{ for } n \geq 3,\\
Q_{4m} &= \langle x, y : x^{2m} = 1, x^m = y^2, y^{-1}xy = x^{-1}\rangle \text{ for } m \geq 2,\\
U_{(n, m)} &= \langle x, y :  x^{2n} = y^m = 1, x^{-1}yx = y^{-1}\rangle \text{ for } m \geq 2 \text{ and } n \geq 2,\\
V_{8n} &= \langle x, y : x^{2n} = y^4 = 1, yx = x^{-1}y^{-1}, y^{-1}x = x^{-1}y \rangle \text{ for } n \geq 2,\\
SD_{8n} &= \langle x, y : x^{4n} = y^2 = 1, yxy = x^{2n -1}\rangle \text{ for } n \geq 2 \text{ and }\\
G(p, m, n) &= \langle x, y : x^{p^m} = y^{p^n} = [x, y]^p = 1, [x, [x, y]] = [y, [x, y]] = 1\rangle, 
\end{align*}
where $p$ is any prime,  $m \geq 1$ and $n \geq 1$.
 
In this paper we compute various spectra and energies of commuting conjugacy class graph of the first five  groups listed above due to the similar nature of their commuting conjugacy class graphs. In a subsequent paper we shall consider commuting conjugacy class graph of $G(p, m, n)$.
Computation of various spectra  is helpful to check whether $\mathcal{CCC}(G)$ is super integral. Recall that a graph $\mathcal{G}$ is called super integral if it is integral,  L-integral and  Q-integral.  In the year 1974, Harary and  Schwenk \cite{hS74} introduced the concept of integral graphs. Several results on these graphs can be found in  \cite{Abreu08,bCrS03,Simic07,Kirkland07,Merries94,Simic08}. It is observed that $\mathcal{CCC}(G)$ is super integral for the groups mentioned above. In Section 4, using the energies computed in Section 3, we determine whether the inequalities in \cite[Conjecture 1]{dBN2020} and \cite[Question 1]{dBN2020} satisfy for     $\mathcal{CCC}(G)$. In Section 5, we determine whether $\mathcal{CCC}(G)$ is hyperenergetic, borderenergetic, L-hyperenergetic, L-borderenergetic,   Q-hyperenergetic    or Q-borderenergetic. 
It is worth mentioning that various spectra and energies of commuting graphs of  finite groups have been computed in \cite{Dutta16,DN16,nath17,dN2018,dBN2020}. 

\section{Definitions and useful results} 
Let $A({\mathcal{G}})$ and $D({\mathcal{G}})$ denote the adjacency matrix  and degree matrix of a graph  ${\mathcal{G}}$ respectively. Then the Laplacian matrix and   Signless Laplacian matrix of ${\mathcal{G}}$ are given by  $L({\mathcal{G}})  = D({\mathcal{G}}) - A({\mathcal{G}})$ and    $Q({\mathcal{G}}) = D({\mathcal{G}}) + A({\mathcal{G}})$ respectively. We write  $\spec({\mathcal{G}})$, $\L-spec({\mathcal{G}})$ and $\Q-spec({\mathcal{G}})$ to denote the spectrum, Laplacian spectrum and Signless Laplacian spectrum of ${\mathcal{G}}$. Also, $\spec({\mathcal{G}}) = \{\alpha_1^{a_1}, \alpha_2^{a_2}, \dots, \alpha_l^{a_l}\}$, $\L-spec({\mathcal{G}}) = \{\beta_1^{b_1}, \beta_2^{b_2}, \dots, \beta_m^{b_m}\}$ and $\Q-spec({\mathcal{G}}) = \{\gamma_1^{c_1}, \gamma_2^{c_2}, \dots, \gamma_n^{c_n}\}$ where $\alpha_1,  \alpha_2, \dots, \alpha_n$ are the eigenvalues of   $A({\mathcal{G}})$ with multiplicities $a_1, a_2, \dots, a_l$;   $\beta_1,  \beta_2, \dots, \beta_m$ are the eigenvalues of  $L({\mathcal{G}})$ with multiplicities $b_1, b_2, \dots, b_m$ and  $\gamma_1, \gamma_2, \dots, \gamma_n$ are the eigenvalues of   $Q({\mathcal{G}})$ with multiplicities $c_1, c_2, \dots, c_n$ respectively. A graph $\mathcal{G}$ is called  integral or L-integral or  Q-integral according as  $\spec({\mathcal{G}})$ or $\L-spec({\mathcal{G}})$ or $\Q-spec({\mathcal{G}})$ contains only integers. Following theorem is helpful in computing various spectra.


\begin{theorem}\label{prethm1}
If $\mathcal{G} = l_1K_{m_1}\sqcup l_2K_{m_2}$, where $l_iK_{m_i}$ denotes the disjoint union of $l_i$ copies of the complete graph  $K_{m_i}$ on $m_i$ vertices for $i= 1, 2$, then 
\begin{align*}
\spec({\mathcal{G}}) &= \left\{(-1)^{\sum_{i=1}^2 l_i(m_i - 1)}, (m_1 - 1)^{l_1},  (m_2 - 1)^{l_2}\right\}\\
\L-spec({\mathcal{G}}) &= \left\{0^{l_1 + l_2}, m_1^{l_1(m_1 - 1)}, m_2^{l_2(m_2 - 1)}\right\} \text{ and}\\
\Q-spec({\mathcal{G}}) &= \Big\lbrace(2m_1 - 2)^{l_1}, (m_1 - 2)^{l_1(m_1 - 1)}, (2m_2 - 2)^{l_2}, (m_2 - 2)^{l_2(m_2 - 1)}\Big\rbrace.
\end{align*}
\end{theorem}

 Depending on the various spectra of a graph, there are various energies called {\it energy, Laplacian energy} and {\it Signless Laplacian energy} denoted by $E({\mathcal{G}})$, $LE({\mathcal{G}})$ and $LE^+({\mathcal{G}})$ respectively. These energies are defined as follows:
\begin{equation}\label{energy}
E({\mathcal{G}})=\sum_{\lambda\in \spec({\mathcal{G}})}|\lambda|,
\end{equation}
\begin{equation}\label{L-energy}
LE({\mathcal{G}})= \sum_{\mu\in \L-spec({\mathcal{G}})}\left|\mu- \frac{2|e(\Gamma)|}{|V(\Gamma)|}\right|,
\end{equation}
\begin{equation}\label{Q-energy}
LE^+({\mathcal{G}})= \sum_{\nu\in \Q-spec({\mathcal{G}})}\left|\nu- \frac{2|e(\Gamma)|}{|V({\mathcal{G}})|}\right|,
\end{equation}
where $V({\mathcal{G}})$ and $e({\mathcal{G}})$ denote the set of vertices and edges of $\Gamma$. 

In 2008, Gutman et al. \cite{gavbr} posed the following conjecture comparing $E({\mathcal{G}})$ and $LE({\mathcal{G}})$.
\begin{conjecture}\label{ai}
(E-LE Conjecture) $E({\mathcal{G}}) \leq LE({\mathcal{G}})$ for any graph ${\mathcal{G}}$.
\end{conjecture}
\noindent However, in the same year, Stevanovi\'{c} et al. \cite{ssm} disproved the above conjecture. In 2009, Liu and Lin \cite{ll} also disproved Conjecture \ref{ai} by providing some counter examples. Following Gutman et al. \cite{gavbr}, recently Dutta et al. have posed the following question in \cite{dBN2020} comparing Laplacian and singless Laplacian energies of graphs.
\begin{ques}\label{ab}
Is $LE({\mathcal{G}})\leq LE^+({\mathcal{G}})$ for all graphs ${\mathcal{G}}$?
\end{ques}

\section{Various  spectra and energies} 
In this section we compute various spectra and energies of commuting conjugacy class graphs of the groups mentioned in the introduction.
\begin{theorem}\label{CCC(G)-D-2n}
If $G = D_{2n}$ then
\begin{enumerate}
\item $\spec(\mathcal{CCC}(G)) = \begin{cases}\left\{(-1)^{\frac{n - 3}{2}}, 0^{1},  \left(\frac{n - 3}{2}\right)^{1}\right\}, & \text{ if $n$  is odd}\\
\left\{(-1)^{\frac{n}{2} - 2}, 0^{2},  \left(\frac{n}{2} - 2\right)^{1}\right\}, & \text{ if $n$ and $\frac{n}{2}$ are even}\\
\left\{(-1)^{\frac{n}{2} - 1}, 1^{1},  \left(\frac{n}{2} - 2\right)^{1}\right\}, & \text{ if $n$ is even and $\frac{n}{2}$ is odd} 
\end{cases}$ 

and $E(\mathcal{CCC}(G))= \begin{cases} n - 3, & \text{ if $n$  is odd}\\
 n - 4, & \text{ if $n$ and $\frac{n}{2}$ are even}\\
 n - 2, & \text{ if $n$ is even and $\frac{n}{2}$ is odd.} 
\end{cases}$
\item $\L-spec(\mathcal{CCC}(G)) = \begin{cases}\left\{0^{2}, \left(\frac{n - 1}{2}\right)^{\frac{n - 3}{2}}\right\}, & \text{ if $n$  is odd}\\
\left\{0^{3}, \left(\frac{n}{2}- 1\right)^{\frac{n}{2} - 2}\right\}, & \text{ if $n$ and $\frac{n}{2}$ are even}\\
\left\{0^{2}, 2^1, \left(\frac{n}{2}- 1\right)^{\frac{n}{2} - 2}\right\}, & \text{ if $n$ is even and $\frac{n}{2}$ is odd} 
\end{cases}$ 

and $LE(\mathcal{CCC}(G))= \begin{cases} 
\frac{2(n - 1)(n - 3)}{n + 1}, & \text{ if $n$  is odd}\\
\frac{3(n - 2)(n - 4)}{n + 2}, & \text{ if $n$ and $\frac{n}{2}$ are even}\\
 4, & \text{ if $n = 6$}\\
\frac{(n - 4)(3n - 10)}{n + 2}, & \text{ if $n$ is even, $n \geq 10$ and $\frac{n}{2}$ is odd.} 
\end{cases}$
\item $\Q-spec(\mathcal{CCC}(G)) = \begin{cases}\left\{0^{1}, (n - 3)^{1}, \left(\frac{n - 5}{2}\right)^{\frac{n - 3}{2}}\right\}, & \text{ if $n$  is odd}\\
\left\{0^{2}, (n - 4)^{1}, \left(\frac{n}{2}- 3\right)^{\frac{n}{2} - 2}\right\}, & \text{ if $n$ and $\frac{n}{2}$ are even}\\
\left\{2^1, 0^{1}, (n - 4)^{1}, \left(\frac{n}{2}- 3\right)^{\frac{n}{2} - 2}\right\}, & \text{ if $n$ is even and $\frac{n}{2}$ is odd} 
\end{cases}$ 

and $LE^+(\mathcal{CCC}(G))= \begin{cases} 
\frac{(n - 3)(n + 3)}{n + 1}, & \text{ if $n$  is odd}\\
\frac{(n - 4)(n + 6)}{n + 2}, & \text{ if } n = 4, 8\\
\frac{2(n - 2)(n - 4)}{n + 2}, & \text{ if $n, \frac{n}{2}$ are even and $n \geq 12$}\\
4, & \text{ if $n = 6$}\\
\frac{22}{3}, & \text{ if $n = 10$}\\
\frac{2(n - 2)(n - 6)}{n + 2}, & \text{ if $n$ is even, $n \geq 14$ and $\frac{n}{2}$ is odd.} 
\end{cases}$
\end{enumerate}
\end{theorem}
\begin{proof}
We shall prove the result by considering the following cases.

\noindent \textbf{Case 1.} $n$ is odd.

 By \cite[Proposition 2.1]{sA2020} we have $\mathcal{CCC}(G) = K_1 \sqcup K_{\frac{n - 1}{2}}$.  Therefore, by Theorem \ref{prethm1}, it follows that
\begin{align*}
&\spec(\mathcal{CCC}(G)) = \left\{(-1)^{\frac{n - 3}{2}}, 0^{1},  \left(\frac{n - 3}{2}\right)^{1}\right\}, \quad  \L-spec(\mathcal{CCC}(G)) = \left\{0^{2}, \left(\frac{n - 1}{2}\right)^{\frac{n - 3}{2}}\right\} \\
\text{ and } &\Q-spec(\mathcal{CCC}(G)) = \left\{0^{1}, (n - 3)^{1}, \left(\frac{n - 5}{2}\right)^{\frac{n - 3}{2}}\right\}.
\end{align*}
Hence,  by \eqref{energy},  we get 
\[
E(\mathcal{CCC}(G))= \frac{n - 3}{2} + \frac{n - 3}{2} = n - 3.
\]

We have $|V(\mathcal{CCC}(G))| = \frac{n + 1}{2}$ and $|e(\mathcal{CCC}(G))| = \frac{(n - 1)(n - 3)}{8}$. Therefore, $\frac{2|e(\mathcal{CCC}(G))|}{|V(\mathcal{CCC}(G))|} = \frac{(n - 1)(n - 3)}{2(n + 1)}$.
Also, 
\[
\left| 0 - \frac{2|e(\mathcal{CCC}(G))|}{|V(\mathcal{CCC}(G))|}\right| = \left| 0 - \frac{(n - 1)(n - 3)}{2(n + 1)}\right| = \frac{(n - 1)(n - 3)}{2(n + 1)} \quad \text{ and }
\]
\[ 
\left| \frac{n - 1}{2} - \frac{2|e(\mathcal{CCC}(G))|}{|V(\mathcal{CCC}(G))|}\right| = \left| \frac{n - 1}{2} - \frac{(n - 1)(n - 3)}{2(n + 1)}\right|  = \frac{2(n - 1)}{n + 1}. 
\]
Now, by \eqref{L-energy},  we have
\[
LE(\mathcal{CCC}(G))= 2\times \frac{(n - 1)(n - 3)}{2(n + 1)} + \frac{n - 3}{2} \times \frac{2(n - 1)}{n + 1} = \frac{2(n - 1)(n - 3)}{n + 1}. 
\]

Again,
\[
\left| n - 3 - \frac{2|e(\mathcal{CCC}(G))|}{|V(\mathcal{CCC}(G))|}\right| = \left| n - 3 - \frac{(n - 1)(n - 3)}{2(n + 1)}\right| =  \frac{(n - 3)(n + 3)}{2(n + 1)} \quad \text{ and }
\]
\[
\left| \frac{n - 5}{2} - \frac{2|e(\mathcal{CCC}(G))|}{|V(\mathcal{CCC}(G))|}\right| = \left|\frac{n - 5}{2} - \frac{(n - 1)(n - 3)}{2(n + 1)}\right|  = \left|\frac{- 4}{n + 1}\right| = \frac{4}{n + 1}.
\]
By \eqref{Q-energy},  we have
\begin{align*}
LE^+(\mathcal{CCC}(G)) &= \frac{(n - 1)(n - 3)}{2(n + 1)} + \frac{(n - 3)(n + 3)}{2(n + 1)} + \frac{n - 3}{2}\times \frac{4}{n + 1} = \frac{(n - 3)(n + 3)}{n + 1}.
\end{align*}


\noindent \textbf{Case 2.} $n$ is even.

Consider  the following  subcases.

\noindent \textbf{Subcase 2.1} $\frac{n}{2}$ is even.

 By \cite[Proposition 2.1]{sA2020} we have $\mathcal{CCC}(G) = 2K_1 \sqcup K_{\frac{n}{2} - 1}$.  Therefore, by Theorem \ref{prethm1}, it follows that
\begin{align*}
&\spec(\mathcal{CCC}(G)) = \left\{(-1)^{\frac{n}{2} - 2}, 0^{2},  \left(\frac{n}{2} - 2\right)^{1}\right\}, \quad  \L-spec(\mathcal{CCC}(G)) = \left\{0^{3}, \left(\frac{n}{2}- 1\right)^{\frac{n}{2} - 2}\right\} \\
\text{ and } &\Q-spec(\mathcal{CCC}(G)) = \left\{0^{2}, (n - 4)^{1}, \left(\frac{n}{2}- 3\right)^{\frac{n}{2} - 2}\right\}.
\end{align*} 
Hence,  by \eqref{energy}, we get 
\[
E(\mathcal{CCC}(G))=  {\frac{n}{2} - 2} + {\frac{n}{2} - 2} = n - 4.
\]

We have $|V(\mathcal{CCC}(G))| = {\frac{n}{2} + 1}$ and $|e(\mathcal{CCC}(G))| = \frac{(n - 2)(n - 4)}{8}$. So, $\frac{2|e(\mathcal{CCC}(G))|}{|V(\mathcal{CCC}(G))|} = \frac{(n - 2)(n - 4)}{2(n + 2)}$.

\noindent Also, 
\[
\left| 0 - \frac{2|e(\mathcal{CCC}(G))|}{|V(\mathcal{CCC}(G))|}\right| = \left| 0 - \frac{(n - 2)(n - 4)}{2(n + 2)}\right| = \frac{(n - 2)(n - 4)}{2(n + 2)} \quad \text{ and }
\]
\[
\left| \frac{n}{2}- 1 - \frac{2|e(\mathcal{CCC}(G))|}{|V(\mathcal{CCC}(G))|}\right| =  \left| \frac{n}{2}- 1 - \frac{(n - 2)(n - 4)}{2(n + 2)}\right| = \frac{3(n - 2)}{n + 2}.
\]
\noindent Now, by \eqref{L-energy},  we have
\[
LE(\mathcal{CCC}(G))= 3 \times \frac{(n - 2)(n - 4)}{2(n + 2)} +  \left(\frac{n}{2}- 2\right) \times \frac{3(n - 2)}{n + 2} = \frac{3(n - 2)(n - 4)}{n + 2}.
\] 

Again,
\[
\left|n - 4 - \frac{2|e(\mathcal{CCC}(G))|}{|V(\mathcal{CCC}(G))|}\right|  = \left|n - 4 - \frac{(n - 2)(n - 4)}{2(n + 2)}\right| = \frac{(n - 4)(n + 6)}{2(n + 2)} \quad \text{ and }
\]
\begin{align*}
\left|\frac{n}{2}- 3 - \frac{2|e(\mathcal{CCC}(G))|}{|V(\mathcal{CCC}(G))|}\right|  = \left|\frac{n}{2}- 3 - \frac{(n - 2)(n - 4)}{2(n + 2)}\right|
= \left|\frac{n - 10}{n + 2}\right| = \begin{cases}
\frac{-n + 10}{n + 2}, & \text{ if } n = 4, 8\\
\frac{n - 10}{n + 2},  & \text{ if } n \geq 12.
\end{cases}
\end{align*}

\noindent By \eqref{Q-energy},  we have
\[
LE^+(\mathcal{CCC}(G))= 2 \times \frac{(n - 2)(n - 4)}{2(n + 2)} + \frac{(n - 4)(n + 6)}{2(n + 2)} + \left(\frac{n}{2}- 2\right) \times \frac{-n + 10}{n + 2} = \frac{(n - 4)(n + 6)}{n + 2},
\]
if $n = 4, 8$. If $n \geq 12$ then
\[
LE^+(\mathcal{CCC}(G))= 2 \times \frac{(n - 2)(n - 4)}{2(n + 2)} + \frac{(n - 4)(n + 6)}{2(n + 2)} + \left(\frac{n}{2}- 2\right) \times \frac{n - 10}{n + 2} = \frac{2(n - 2)(n - 4)}{n + 2}.
\]


\noindent \textbf{Subcase 2.2} $\frac{n}{2}$ is odd.

 By \cite[Proposition 2.1]{sA2020} we have $\mathcal{CCC}(G) = K_2 \sqcup K_{\frac{n}{2} - 1}$.  Therefore, by Theorem \ref{prethm1}, it follows that
\begin{align*}
&\spec(\mathcal{CCC}(G)) = \left\{(-1)^{\frac{n}{2} - 1}, 1^{1},  \left(\frac{n}{2} - 2\right)^{1}\right\}, \quad  \L-spec(\mathcal{CCC}(G)) = \left\{0^{2}, 2^1, \left(\frac{n}{2}- 1\right)^{\frac{n}{2} - 2}\right\} \\
\text{ and } &\Q-spec(\mathcal{CCC}(G)) = \left\{2^1, 0^{1}, (n - 4)^{1}, \left(\frac{n}{2}- 3\right)^{\frac{n}{2} - 2}\right\}.
\end{align*} 
Hence, by \eqref{energy}, we get
\[
E(\mathcal{CCC}(G))= \frac{n}{2} - 1 + 1 + \frac{n}{2} - 2 = n - 2.
\]

We have $|V(\mathcal{CCC}(G))| = \frac{n}{2} + 1$ and $|e(\mathcal{CCC}(G))| = \frac{(n - 2)(n - 4) + 8}{8}$. Therefore, $\frac{2|e(\mathcal{CCC}(G))|}{|V(\mathcal{CCC}(G))|} = \frac{(n - 2)(n - 4) + 8}{2(n + 2)}$. Also, 
\[
\left|0 - \frac{2|e(\mathcal{CCC}(G))|}{|V(\mathcal{CCC}(G))|}\right| = \left|0 - \frac{(n - 2)(n - 4) + 8}{2(n + 2)}\right| = \frac{(n - 2)(n - 4) + 8}{2(n + 2)},
\]
\[
\left|2 - \frac{2|e(\mathcal{CCC}(G))|}{|V(\mathcal{CCC}(G))|}\right| = \left|2 - \frac{(n - 2)(n - 4) + 8}{2(n + 2)}\right| = \left|\frac{-n^2 + 10n - 8}{2(n + 2)}\right| = \begin{cases}
1, & \text{ if } n = 6\\
\frac{n^2 - 10n + 8}{2(n + 2)}, & \text{ if } n \geq 10
\end{cases} 
\]
\[
\text{ and } \left|\frac{n}{2} -  1 - \frac{2|e(\mathcal{CCC}(G))|}{|V(\mathcal{CCC}(G))|}\right| = \left|\frac{n}{2} -  1 - \frac{(n - 2)(n - 4) + 8}{2(n + 2)}\right| = \frac{3n - 10}{n + 2}.
\]
\noindent Now, by \eqref{L-energy},  we have
\[
LE(\mathcal{CCC}(G))= 2 \times \frac{(n - 2)(n - 4) + 8}{2(n + 2)} + 1 + \left(\frac{n}{2}- 2\right) \times \frac{3n - 10}{n + 2} = 4,
\]  
if $n = 6$. If $n \geq 10$ then
\begin{align*}
LE(\mathcal{CCC}(G)) & = 2 \times \frac{(n - 2)(n - 4) + 8}{2(n + 2)} + \frac{n^2 - 10n + 8}{2(n + 2)} + \left(\frac{n}{2}- 2\right) \times \frac{3n - 10}{n + 2} \\
& = \frac{3n^2 - 22n - 40}{n + 2} = \frac{(n - 4)(3n - 10)}{n + 2}.
\end{align*}

Again,
\[
\left|n - 4 - \frac{2|e(\mathcal{CCC}(G))|}{|V(\mathcal{CCC}(G))|}\right| = \left|n - 4 - \frac{(n - 2)(n - 4) + 8}{2(n + 2)}\right| = \frac{n^2 + 2n - 32}{2(n + 2)} \quad \text{ and }
\]
\[
\left|\frac{n}{2} - 3 - \frac{2|e(\mathcal{CCC}(G))|}{|V(\mathcal{CCC}(G))|}\right| = \left|\frac{n}{2} - 3 - \frac{(n - 2)(n - 4) + 8}{2(n + 2)}\right| = \left|\frac{n - 14}{n + 2}\right| = \begin{cases}
\frac{-n + 14}{n + 2}, &\text{ if } n = 6, 10\\
\frac{n - 14}{n + 2}, &\text{ if } n \geq 14.
\end{cases}
\]

\noindent By \eqref{Q-energy},  we have
\[
LE^+(\mathcal{CCC}(G))= 1 + \frac{(n - 2)(n - 4) + 8}{2(n + 2)} +
 \frac{n^2 + 2n - 32}{2(n + 2)} + \left(\frac{n}{2}- 2\right) \times \frac{-n + 14}{n + 2} = 4,
 \]
if $n = 6$. If $n = 10$ then 
\[
LE^+(\mathcal{CCC}(G))= \frac{n^2 - 10n + 8}{2(n + 2)}  + \frac{(n - 2)(n - 4) + 8}{2(n + 2)} +
 \frac{n^2 + 2n - 32}{2(n + 2)} + \left(\frac{n}{2}- 2\right) \times \frac{-n + 14}{n + 2} = \frac{22}{3}.
 \]
If $n \geq 14$ then
\begin{align*}
LE^+(\mathcal{CCC}(G)) &= \frac{n^2 - 10n + 8}{2(n + 2)}  + \frac{(n - 2)(n - 4) + 8}{2(n + 2)} + \frac{n^2 + 2n - 32}{2(n + 2)} + \left(\frac{n}{2}- 2\right) \times \frac{n - 14}{n + 2}\\
&= \frac{2(n - 2)(n - 6)}{n + 2}.
\end{align*}
This completes the proof. 
\end{proof}
\begin{theorem}\label{Q(4m)}
If $G = Q_{4m}$ then
\begin{enumerate}
\item $\spec(\mathcal{CCC}(G)) = \begin{cases}
\left\{(-1)^{m - 1}, 1^{1},  (m - 2)^{1}\right\}, &\text{ if $m$  is odd}\\
\left\{(-1)^{m - 2}, 0^{2},  (m - 2)^{1}\right\}, &\text{ if $m$  is  even} 
\end{cases}$ 

and $E(\mathcal{CCC}(G))= \begin{cases} 
2m - 2, &\text{ if $m$  is odd}\\
2m - 4, &\text{ if $m$  is  even.} 
\end{cases}$
\item $\L-spec(\mathcal{CCC}(G)) = \begin{cases}
\left\{0^{2}, 2^1,  (m - 1)^{m - 2}\right\}, &\text{ if $m$  is odd}\\
\left\{0^{3},   (m - 1)^{m - 2}\right\}, &\text{ if $m$  is  even} 
\end{cases}$ 

and $LE(\mathcal{CCC}(G))= \begin{cases} 
4, & \text{ if $m = 3$}\\
\frac{2(m - 2)(3m - 5)}{m + 1}, & \text{ if $m$  is odd and $m \geq 5$}\\
\frac{6(m - 1)(m - 2)}{m + 1}, & \text{ if $m$  is  even.} 
\end{cases}$
\item $\Q-spec(\mathcal{CCC}(G)) = \begin{cases}
\left\{2^1, 0^{1}, (2m - 4)^{1}, (m - 3)^{m - 2}\right\}, &\text{ if $m$  is odd}\\
\left\{0^{2}, (2m - 4)^{1}, (m - 3)^{m - 2}\right\}, &\text{ if $m$  is  even.} 
\end{cases}$ 

and $LE^+(\mathcal{CCC}(G))= \begin{cases} 
4, & \text{ if $m = 3$}\\
\frac{22}{3}, & \text{ if $m = 5$}\\
\frac{4(m - 1)(m - 3)}{m + 1}, & \text{ if $m$  is odd and $m \geq 7$}\\
\frac{2(m - 2)(m + 3)}{m + 1}, & \text{ if $m = 2, 4$} \\
\frac{4(m - 1)(m - 2)}{m + 1}, & \text{ if $m$  is  even and $m \geq 6$.}
\end{cases}$
\end{enumerate}
\end{theorem}
\begin{proof}
We shall prove the result by considering the following cases.

\noindent \textbf{Case 1.} $m$ is odd.

 By \cite[Proposition 2.2]{sA2020} we have $\mathcal{CCC}(G) = K_2 \sqcup K_{m - 1}$.  Therefore, by Theorem \ref{prethm1}, it follows that
\begin{align*}
&\spec(\mathcal{CCC}(G)) = \left\{(-1)^{m - 1}, 1^{1},  (m - 2)^{1}\right\}, \quad  \L-spec(\mathcal{CCC}(G)) = \left\{0^{2}, 2^1,  (m - 1)^{m - 2}\right\} \\
\text{ and } &\Q-spec(\mathcal{CCC}(G)) = \left\{2^1, 0^{1}, (2m - 4)^{1}, (m - 3)^{m - 2}\right\}.
\end{align*} 
Hence, by \eqref{energy}, we get 
\[
E(\mathcal{CCC}(G))= m - 1 + 1 + m - 2 = 2m - 2.
\]

We have $|V(\mathcal{CCC}(G))| = m + 1$ and $|e(\mathcal{CCC}(G))| = \frac{(m - 1)(m - 2) + 2}{2}$. Therefore, $\frac{2|e(\mathcal{CCC}(G))|}{|V(\mathcal{CCC}(G))|} = \frac{(m - 1)(m - 2) + 2}{m + 1}$. Also, 
\[
\left|0 - \frac{2|e(\mathcal{CCC}(G))|}{|V(\mathcal{CCC}(G))|}\right| = \left|0 - \frac{(m - 1)(m - 2) + 2}{m + 1}\right| = \frac{(m - 1)(m - 2) + 2}{m + 1},
\]
\[
\left|2 - \frac{2|e(\mathcal{CCC}(G))|}{|V(\mathcal{CCC}(G))|}\right| = \left|2 - \frac{(m - 1)(m - 2) + 2}{m + 1}\right| = \left|\frac{-m^2 + 5m - 2}{m + 1}\right| = \begin{cases}
1, & \text{ if } m = 3\\
\frac{m^2 - 5m + 2}{m + 1}, & \text{ if } m \geq 5
\end{cases}
\]
\[
\text{ and } \left|m - 1 - \frac{2|e(\mathcal{CCC}(G))|}{|V(\mathcal{CCC}(G))|}\right| = \left|m - 1 - \frac{(m - 1)(m - 2) + 2}{m + 1}\right| = \frac{3m - 5}{m + 1}.
\]

\noindent Now, by \eqref{L-energy},  we have
\[
LE(\mathcal{CCC}(G))= 2\times \frac{(m - 1)(m - 2) + 2}{m + 1} + 1 + (m - 2) \times \frac{3m - 5}{m + 1} = 4,
\] 
if $m = 3$. If $m \geq 5$ then
\[
LE(\mathcal{CCC}(G))= 2\times \frac{(m - 1)(m - 2) + 2}{m + 1} + \frac{m^2 - 5m + 2}{m + 1} + (m - 2) \times \frac{3m - 5}{m + 1} =  \frac{2(m - 2)(3m - 5)}{m + 1}.
\] 

Again,
\[
\left|2m - 4 - \frac{2|e(\mathcal{CCC}(G))|}{|V(\mathcal{CCC}(G))|}\right| = \left|2m - 4 - \frac{(m - 1)(m - 2) + 2}{m + 1}\right| = \frac{m^2 + m - 8}{m + 1} \quad \text{ and }
\]
\[
\left|m - 3 - \frac{2|e(\mathcal{CCC}(G))|}{|V(\mathcal{CCC}(G))|}\right| = \left|m - 3 - \frac{(m - 1)(m - 2) + 2}{m + 1}\right| = \left|\frac{m - 7}{m + 1}\right| = \begin{cases}
\frac{-m + 7}{m + 1}, & \text{ if $m = 3, 5$}\\
\frac{m - 7}{m + 1},  & \text{ if $m \geq 7$}.
\end{cases}
\]

\noindent By \eqref{Q-energy},  we have
\[
LE^+(\mathcal{CCC}(G))= 1 + \frac{(m - 1)(m - 2) + 2}{m + 1} + \frac{m^2 + m - 8}{m + 1} + (m - 2)\times \frac{-m + 7}{m + 1} = 4,
\]
if $m = 3$. If $m = 5$ then 
\[
LE^+(\mathcal{CCC}(G))= \frac{m^2 - 5m + 2}{m + 1} + \frac{(m - 1)(m - 2) + 2}{m + 1} + \frac{m^2 + m - 8}{m + 1} + (m - 2)\times \frac{-m + 7}{m + 1} = \frac{22}{3}. 
\]
If $m \geq 7$ then
\begin{align*}
LE^+(\mathcal{CCC}(G)) &= \frac{m^2 - 5m + 2}{m + 1} + \frac{(m - 1)(m - 2) + 2}{m + 1} + \frac{m^2 + m - 8}{m + 1} + (m - 2)\times \frac{m - 7}{m + 1} \\
&= \frac{4(m - 1)(m - 3)}{m + 1}. 
\end{align*}

\noindent \textbf{Case 2.} $m$ is even.

 By \cite[Proposition 2.2]{sA2020} we have $\mathcal{CCC}(G) = 2K_1 \sqcup K_{m - 1}$.  Therefore, by Theorem \ref{prethm1}, it follows that
\begin{align*}
&\spec(\mathcal{CCC}(G)) = \left\{(-1)^{m - 2}, 0^{2},  (m - 2)^{1}\right\}, \quad  \L-spec(\mathcal{CCC}(G)) = \left\{0^{3},   (m - 1)^{m - 2}\right\} \\
\text{ and } &\Q-spec(\mathcal{CCC}(G)) = \left\{0^{2}, (2m - 4)^{1}, (m - 3)^{m - 2}\right\}.
\end{align*}
Hence, by \eqref{energy}, we get
\[
E(\mathcal{CCC}(G))= m - 2 + m - 2 = 2m - 4.
\]

We have $|V(\mathcal{CCC}(G))| = m + 1$ and $|e(\mathcal{CCC}(G))| = \frac{(m - 1)(m - 2)}{2}$. Therefore, $\frac{2|e(\mathcal{CCC}(G))|}{|V(\mathcal{CCC}(G))|} = \frac{(m - 1)(m - 2)}{m + 1}$. Also, 
\[
\left|0 - \frac{2|e(\mathcal{CCC}(G))|}{|V(\mathcal{CCC}(G))|}\right| = \left|0 - \frac{(m - 1)(m - 2)}{m + 1}\right| = \frac{(m - 1)(m - 2)}{m + 1} \quad \text{ and }
\]
\[
\left|m - 1 - \frac{2|e(\mathcal{CCC}(G))|}{|V(\mathcal{CCC}(G))|}\right| = \left|m - 1 - \frac{(m - 1)(m - 2)}{m + 1}\right| = \frac{3(m - 1)}{m + 1}.
\]

\noindent Now, by \eqref{L-energy},  we have
\[
LE(\mathcal{CCC}(G))= 3 \times \frac{(m - 1)(m - 2)}{m + 1} + (m - 2) \times\frac{3(m - 1)}{m + 1} = \frac{6(m - 1)(m - 2)}{m + 1}.
\] 

Again,
\[
\left|2m - 4 - \frac{2|e(\mathcal{CCC}(G))|}{|V(\mathcal{CCC}(G))|}\right| = \left|2m - 4 - \frac{(m - 1)(m - 2)}{m + 1}\right| = \frac{(m - 2)(m + 3)}{m + 1} \quad \text{ and }
\]
\[
\left|m - 3 - \frac{2|e(\mathcal{CCC}(G))|}{|V(\mathcal{CCC}(G))|}\right| = \left|m - 3 - \frac{(m - 1)(m - 2)}{m + 1}\right| =  \left|\frac{m - 5}{m + 1}\right| = \begin{cases}
\frac{-m + 5}{m + 1}, & \text{ if } m = 2, 4\\
\frac{m - 5}{m + 1},  & \text{ if } m \geq 6.
\end{cases}
\]

\noindent By \eqref{Q-energy},  we have
\begin{align*}
LE^+(\mathcal{CCC}(G)) &= 2\times \frac{(m - 1)(m - 2)}{m + 1} + \frac{(m - 2)(m + 3)}{m + 1} + (m - 2) \times \frac{-m + 5}{m + 1} \\
&= \frac{2(m - 2)(m + 3)}{m + 1}, 
\end{align*}
if $m = 2, 4$. If $m \geq 6$ then
\begin{align*}
LE^+(\mathcal{CCC}(G)) &= 2\times \frac{(m - 1)(m - 2)}{m + 1} + \frac{(m - 2)(m + 3)}{m + 1} + (m - 2) \times  \frac{m - 5}{m + 1}\\
&= \frac{4(m - 1)(m - 2)}{m + 1}.
\end{align*}
This completes the proof.
\end{proof}

\begin{theorem}\label{U(n,m)}
If $G = U_{(n, m)}$ then
\begin{enumerate}
\item $\spec(\mathcal{CCC}(G)) = \begin{cases}
\left\{(-1)^{\frac{n(m + 1) - 4}{2}},   \left(\frac{n(m - 1) - 2}{2}\right)^{1}, (n - 1)^1\right\}, &\text{ if $m$  is odd and $n \geq 2$}\\
\left\{(-1)^{\frac{n(m + 2) - 6}{2}},   \left(\frac{n(m - 2) - 2}{2}\right)^{1}, (n - 1)^2\right\}, &\text{ if $m$  is  even and $n \geq 2$} 
\end{cases}$ 

and $E(\mathcal{CCC}(G))= \begin{cases} 
n(m + 1) - 4, & \text{ if $m$  is odd and $n \geq 2$}\\
4(n - 1), & \text{ if $m = 2$ and $n \geq 2$}\\
n(m + 2) - 6, & \text{ if $m$  is  even, $m \geq 4$ and $n \geq 2$.} 
\end{cases}$
\item $\L-spec(\mathcal{CCC}(G)) = \begin{cases}
\left\{0^{2}, \left(\frac{n(m - 1)}{2}\right)^{\frac{n(m - 1) - 2}{2}}, n^{n - 1}\right\}, &\text{ if $m$  is odd and $n \geq 2$}\\
\left\{0^{3}, \left(\frac{n(m - 2)}{2}\right)^{\frac{n(m - 2) - 2}{2}}, n^{2n - 2}\right\}, &\text{ if $m$  is  even and $n \geq 2$} 
\end{cases}$ 

and $LE(\mathcal{CCC}(G))= \begin{cases} 
4(n - 1), & \text{ if  $m = 3$ and $n \geq 2$} \\
\frac{2(2n - 1)(n + 3)}{3}, & \text{ if  $m = 5$ and $n \geq 2$} \\
\frac{m^2n^2 - 4mn^2 + m^2n + 3n^2 - 2mn -  2m + 5n - 2}{m  + 1}, & \text{ if $m$  is odd,  $m \geq 7$}\\
 & \text{ and $n \geq 2$}\\
4(n - 1), & \text{ if $m = 2$ and $n \geq 2$}\\
6(n - 1), & \text{ if $m = 4$ and $n \geq 2$}\\
\frac{2m^2n^2 - 12mn^2 + m^2n + 16n^2 - 4mn -  2m + 12n - 4}{m + 2}, & \text{ if $m$  is  even, $m \geq 6$}\\
& \text{ and $n \geq 2$}.
\end{cases}$
\item $\Q-spec(\mathcal{CCC}(G)) = \begin{cases}
&\left\{(n(m - 1) - 2)^{1}, \left(\frac{n(m - 1) - 4}{2}\right)^{\frac{n(m - 1) - 2}{2}},  (2n - 2)^1, (n - 2)^{n - 1}\right\}, \\
&\hspace{6.4cm}\text{if $m$  is odd and $n \geq 2$}\\
&\left\{(n(m - 2) - 2)^{1}, \left(\frac{n(m - 2) - 4}{2}\right)^{\frac{n(m - 2) - 2}{2}},  (2n - 2)^2, (n - 2)^{2n - 2}\right\}, \\
&\hspace{6.2cm}\text{ if $m$  is  even and $n \geq 2$} 
\end{cases}$ 

and $LE^+(\mathcal{CCC}(G))= \begin{cases} 
4(n - 1), & \text{ if $m =3$ and $n \geq 2$}\\
\frac{22}{3}, & \text{ if $m = 5$ and $n = 2$}\\
\frac{2(2n + 3)(n - 1)}{3}, & \text{ if $m = 5$ and $n \geq 3$}\\
\frac{n^2(m - 1)(m - 3)}{m + 1}, & \text{ if $m$  is odd, $m \geq 7$ and $n \geq 2$}\\
4(n - 1), & \text{ if $m = 2$ and $n \geq 2$}\\
6(n - 1), & \text{ if $m = 4$ and $n \geq 2$}\\
2(n + 2)(n - 1), & \text{ if $m = 6$ and $n \geq 2$}\\
\frac{2n^2(m - 2)(m - 4)}{m + 2}, & \text{ if $m$  is  even, $m \geq 8$ and $n \geq 2$.} 
\end{cases}$
\end{enumerate}
\end{theorem}
\begin{proof}
We shall prove the result by considering the following cases.

\noindent \textbf{Case 1.} $m$ is odd.

 By \cite[Proposition 2.3]{sA2020} we have $\mathcal{CCC}(G) =   K_{\frac{n(m - 1)}{2}} \sqcup K_n$.  Therefore, by Theorem \ref{prethm1}, it follows that
\begin{align*}
\spec(\mathcal{CCC}(G)) &= \left\{(-1)^{\frac{n(m + 1) - 4}{2}},   \left(\frac{n(m - 1) - 2}{2}\right)^{1}, (n - 1)^1\right\}, \\
\L-spec(\mathcal{CCC}(G)) &= \left\{0^{2}, \left(\frac{n(m - 1)}{2}\right)^{\frac{n(m - 1) - 2}{2}}, n^{n - 1}\right\} \\
\text{ and } \Q-spec(\mathcal{CCC}(G)) &= \left\{(n(m - 1) - 2)^{1}, \left(\frac{n(m - 1) - 4}{2}\right)^{\frac{n(m - 1) - 2}{2}}, \quad (2n - 2)^1, (n - 2)^{n - 1}\right\}.
\end{align*}
Hence, by \eqref{energy}, we get 
\[
E(\mathcal{CCC}(G))= \frac{n(m + 1) - 4}{2} + \frac{n(m - 1) - 2}{2} + n - 1 = n(m + 1) - 4.
\]

We have $|V(\mathcal{CCC}(G))| = \frac{n(m + 1)}{2}$ and $|e(\mathcal{CCC}(G))| = \frac{n^2(m - 1)^2 - 2n(m - 2n + 1)}{8}$. Therefore, $\frac{2|e(\mathcal{CCC}(G))|}{|V(\mathcal{CCC}(G))|}$ $= \frac{n(m - 1)^2 - 2(m - 2n + 1)}{2(m + 1)}$.
Also, 
\[
\left|0 - \frac{2|e(\mathcal{CCC}(G))|}{|V(\mathcal{CCC}(G))|}\right| = \left|0 - \frac{n(m - 1)^2 - 2(m - 2n + 1)}{2(m + 1)}\right| = \frac{n(m - 1)^2 - 2(m - 2n + 1)}{2(m + 1)},
\]
since $n(m - 1)^2 - 2(m - 2n + 1) = m^2n - 2m(n + 1) + 5n - 2 > 0$;
\[
\left|\frac{n(m - 1)}{2} - \frac{2|e(\mathcal{CCC}(G))|}{|V(\mathcal{CCC}(G))|}\right| = \left|\frac{n(m - 1)}{2} - \frac{n(m - 1)^2 - 2(m - 2n + 1)}{2(m + 1)}\right| = \frac{n(m - 3) + m + 1}{m + 1} \, \text{ and }
\]
\[
\left|n - \frac{2|e(\mathcal{CCC}(G))|}{|V(\mathcal{CCC}(G))|}\right| = \left|n - \frac{n(m - 1)^2 - 2(m - 2n + 1)}{2(m + 1)}\right| = \left|\frac{-f_1(m, n)}{2(m + 1)}\right|, 
\]
where $f_1(m, n) = n(m^2 + 3) - (4mn + 2m + 2)$. For $m = 3$ and $n \geq 2$ we have
$f_1(3, n) = -8$.  For $m = 5$ and $n \geq 2$ we have
$f_1(5, n) = 8n - 12 > 0$. For $m \geq 7$ and $n \geq 2$ we have $m^2 + 3 > m^2 > 4m + 2m + 2$. Therefore, $n(m^2 + 3)  > 4mn + (2m + 2)n > 4mn + 2m + 2$ and so $f_1(m, n) > 0$. Hence,
\[
\left|\frac{-f_1(m, n)}{2(m + 1)}\right| = \begin{cases}
1, & \text{ if  $m = 3$ and $n \geq 2$} \\
\frac{2n - 3}{3}, & \text{ if  $m = 5$ and $n \geq 2$} \\
\frac{n(m^2 + 3) - (4mn + 2m + 2)}{2(m + 1)}, & \text{ if  $m \geq 7$ and $n \geq 2$}.
\end{cases}
\] 
\noindent Now, by \eqref{L-energy},  we have
\begin{align*}
LE(\mathcal{CCC}(G))&= 2 \times \frac{n(m - 1)^2 - 2(m - 2n + 1)}{2(m + 1)} + \frac{n(m - 1) - 2}{2} \times \frac{n(m - 3) + m + 1}{m + 1} \\
& \hspace{5.4cm}  + (n - 1) \times 1\\
& = 4(n - 1), 
\end{align*}
if $m = 3$ and $n \geq 2$. If $m = 5$ and $n \geq 2$ then
\begin{align*}
LE(\mathcal{CCC}(G))&= 2 \times \frac{n(m - 1)^2 - 2(m - 2n + 1)}{2(m + 1)} + \frac{n(m - 1) - 2}{2} \times \frac{n(m - 3) + m + 1}{m + 1} \\
& \hspace{5.4cm}  + (n - 1) \times \frac{2n - 3}{3}\\
& = \frac{2(2n - 1)(n + 3)}{3}. 
\end{align*}

If  $m \geq 7$ and $n \geq 2$ then
\begin{align*}
LE(\mathcal{CCC}(G)) & = 2 \times \frac{n(m - 1)^2 - 2(m - 2n + 1)}{2(m + 1)} + \frac{n(m - 1) - 2}{2} \times \frac{n(m - 3) + m + 1}{m + 1}\\
& \hspace{5.4cm}  + (n - 1) \times \frac{n(m^2 + 3) - (4mn + 2m + 2)}{2(m + 1)}\\
& = \frac{m^2n^2 - 4mn^2 + m^2n + 3n^2 - 2mn -  2m + 5n - 2}{m  + 1}. 
\end{align*}

Again,
\begin{align*}
\left|n(m - 1) - 2 - \frac{2|e(\mathcal{CCC}(G))|}{|V(\mathcal{CCC}(G))|}\right| & = \left|n(m - 1) - 2 - \frac{n(m - 1)^2 - 2(m - 2n + 1)}{2(m + 1)}\right|\\
& = \left|\frac{n(m - 1)(m + 3) - 2(m + 2n + 1)}{2(m + 1)}\right|\\
& = \frac{n(m - 1)(m + 3) - 2(m + 2n + 1)}{2(m + 1)},
\end{align*}
since $n(m - 1)(m + 3) - 2(m + 2n + 1) = n(m^2 - 4) - 2 + n(m - 3) + m(n - 2) > 0$;
\[
\left|\frac{n(m - 1) - 4}{2} - \frac{2|e(\mathcal{CCC}(G))|}{|V(\mathcal{CCC}(G))|}\right| = \left|\frac{n(m - 1) - 4}{2} - \frac{n(m - 1)^2 - 2(m - 2n + 1)}{2(m + 1)}\right| = \left|\frac{f_2(m, n)}{2(m + 1)}\right|,
\]
where $f_2(m, n) = n(m - 6) - 2 + m(n - 2)$. Clearly, for $m \geq 7$ and $n \geq 2$ we have $f_2(m, n)  \geq 0$. For  $m = 3$ and $n \geq 2$ we have 
$f_2(3, n) = - 8$. Also for   $m = 5$ and $n \geq 2$ we have  $f_2(5, n) = 4n - 12$. Therefore,  $f_2(5, 2) = -4$ and $f_2(5, n) \geq 0$ for $n \geq 3$. Hence,
\[
\left|\frac{f_2(m, n)}{2(m + 1)}\right| = \begin{cases}
1, & \text{ if $m = 3$ and $n \geq 2$}\\
\frac{1}{3}, & \text{ if $m = 5$ and $n = 2$}\\
\frac{n - 3}{3}, & \text{ if $m = 5$ and $n \geq 3$}\\
\frac{n(m - 3) - m - 1}{m + 1}, & \text{ if $m \geq 7$ and $n \geq 2$}.
\end{cases}
\] 
\[
\left|2n - 2 - \frac{2|e(\mathcal{CCC}(G))|}{|V(\mathcal{CCC}(G))|}\right| = \left|2n - 2 - \frac{n(m - 1)^2 - 2(m - 2n + 1)}{2(m + 1)}\right| =  \left|-\frac{f_3(m, n)}{2(m + 1)}\right|,
\]
where $f_3(m, n) = mn(m - 6) + 2m + n + 2$. Clearly, $f_3(m, n) > 0$ if $m \geq 7$ and $n \geq 2$. For $m = 3$ and $n \geq 2$ we have $f_3(3, n) = - 8n + 8 < 0$. For $m = 5$ and $n \geq 2$ we have $f_3(5, n) = - 4n + 12$. Therefore, $f_3(5, 2) = 4$ and $f_3(5, n) \leq 0$ if $n \geq 3$. Hence,
\[
\left|-\frac{f_3(m, n)}{2(m + 1)}\right| = \begin{cases}
n - 1, & \text{ if $m =3$ and $n \geq 2$}\\
\frac{1}{3}, & \text{ if $m = 5$ and $n = 2$}\\
\frac{n - 3}{3}, & \text{ if $m = 5$ and $n \geq 3$}\\
\frac{mn(m - 6) + 2m + n + 2}{2(m + 1)}, & \text{ if $m \geq 7$ and $n \geq 2$}.
\end{cases}
\]

\[
\left|n - 2 - \frac{2|e(\mathcal{CCC}(G))|}{|V(\mathcal{CCC}(G))|}\right| = \left|n - 2 - \frac{n(m - 1)^2 - 2(m - 2n + 1)}{2(m + 1)}\right| = \left|-\frac{f_4(m, n)}{2(m + 1)}\right|,
\]
where $f_4(m, n) = mn(m - 2) + 2 - (m(n - 2) + n(m - 3))$. For $m = 3$ and $n \geq 2$ we have $f_4(3, n) = 8$. Also, for $m \geq 5$ and $n \geq 2$ we have
\[
mn(m - 2) - 2mn + 2 = mn(m - 4) + 2 > -2m -3n.
\]
Therefore, 
\[
mn(m - 2) + 2 > 2mn -2m -3n = m(n - 2) + n(m - 3)
\]
and so $f_4(m, n) > 0$ for $m \geq 5$ and $n \geq 2$. Hence,
\[
\left|-\frac{f_4(m, n)}{2(m + 1)}\right| = \frac{f_4(m, n)}{2(m + 1)} =  \frac{mn(m - 2) + 2 - m(n - 2) - n(m - 3)}{2(m + 1)}.
\]

\noindent By \eqref{Q-energy},  we have
\begin{align*}
LE^+(\mathcal{CCC}(G)) &= \frac{n(m - 1)(m + 3) - 2(m + 2n + 1)}{2(m + 1)} + \frac{n(m - 1) - 2}{2} \times 1 + (n - 1)\\
& \hspace{3cm} + (n - 1) \times   \frac{mn(m - 2) + 2 - m(n - 2) - n(m - 3)}{2(m + 1)}\\
& = 4(n - 1),
\end{align*}
if  $m =3$ and $n \geq 2$. If $m = 5$ and $n = 2$ then
\begin{align*}
LE^+(\mathcal{CCC}(G)) &= \frac{n(m - 1)(m + 3) - 2(m + 2n + 1)}{2(m + 1)} + \frac{n(m - 1) - 2}{2} \times \frac{1}{3} + \frac{1}{3} \\
& \hspace{3cm} + (n - 1) \times   \frac{mn(m - 2) + 2 - m(n - 2) - n(m - 3)}{2(m + 1)}\\
& = \frac{22}{3}.
\end{align*}
If $m = 5$ and $n \geq 3$ then
\begin{align*}
LE^+(\mathcal{CCC}(G)) &= \frac{n(m - 1)(m + 3) - 2(m + 2n + 1)}{2(m + 1)} + \frac{n(m - 1) - 2}{2} \times \frac{n - 3}{3} + \frac{n - 3}{3}\\
& \hspace{3cm} + (n - 1) \times   \frac{mn(m - 2) + 2 - m(n - 2) - n(m - 3)}{2(m + 1)}\\
& = \frac{2(2n^2 + n - 3)}{3} = \frac{2(2n + 3)(n - 1)}{3}.
\end{align*}
If $m \geq 7$ and $n \geq 2$ then
\begin{align*}
LE^+(\mathcal{CCC}(G)) &= \frac{n(m - 1)(m + 3) - 2(m + 2n + 1)}{2(m + 1)} + \frac{n(m - 1) - 2}{2} \times \frac{n(m - 3) - m - 1}{m + 1}\\
& + \frac{mn(m - 6) + 2m + n + 2}{2(m + 1)} + (n - 1) \times    \frac{mn(m - 2) + 2 - m(n - 2) - n(m - 3)}{2(m + 1)}\\
& = \frac{n^2(m - 1)(m - 3)}{m + 1}.
\end{align*}

\vspace{1cm}

\noindent \textbf{Case 2.} $m$ is even.

 By \cite[Proposition 2.3]{sA2020} we have $\mathcal{CCC}(G) = K_{\frac{n(m - 2)}{2}} \sqcup 2K_n$.  Therefore, by Theorem \ref{prethm1}, it follows that
\begin{align*}
\spec(\mathcal{CCC}(G)) &= \left\{(-1)^{\frac{n(m + 2) - 6}{2}},   \left(\frac{n(m - 2) - 2}{2}\right)^{1}, (n - 1)^2\right\}, \\
\L-spec(\mathcal{CCC}(G)) &= \left\{0^{3}, \left(\frac{n(m - 2)}{2}\right)^{\frac{n(m - 2) - 2}{2}}, n^{2n - 2}\right\} \\
\text{ and } \Q-spec(\mathcal{CCC}(G)) &= \left\{(n(m - 2) - 2)^{1}, \left(\frac{n(m - 2) - 4}{2}\right)^{\frac{n(m - 2) - 2}{2}}, \quad (2n - 2)^2, (n - 2)^{2n - 2}\right\}.
\end{align*}
We have 
\[
\left|\frac{n(m - 2) - 2}{2}\right| = \begin{cases}
1, & \text{ if } m = 2\\
\frac{n(m - 2) - 2}{2}, & \text{ if } m  \geq 4.
\end{cases}
\] 
Therefore, by \eqref{energy}, we have 
\[
E(\mathcal{CCC}(G))= \frac{n(m + 2) - 6}{2} + 1 + 2(n - 1) = 4(n - 1),
\]
if $m = 2$. If $m \geq 4$ then
\[
E(\mathcal{CCC}(G))= \frac{n(m + 2) - 6}{2} + \frac{n(m - 2) - 2}{2} + 2(n - 1) = n(m + 2) - 6.
\]

We have $|V(\mathcal{CCC}(G))| = \frac{n(m + 2)}{2}$ and $|e(\mathcal{CCC}(G))| = \frac{n^2(m - 2)^2 - 2n(m - 4n + 2)}{8}$. Therefore, $\frac{2|e(\mathcal{CCC}(G))|}{|V(\mathcal{CCC}(G))|}$ $ = \frac{n(m - 2)^2 - 2(m - 4n + 2)}{2(m + 2)}$.
Also, 
\[
\left|0 - \frac{2|e(\mathcal{CCC}(G))|}{|V(\mathcal{CCC}(G))|}\right| = \left|0 - \frac{n(m - 2)^2 - 2(m - 4n + 2)}{2(m + 2)}\right| = \left|\frac{-f_5(m, n)}{2(m + 2)}\right|,
\]
where $f_5(m, n) = m(n(m - 4) - 2) + 12n - 4$. Note that for $m \geq 6$ we have  $f_5(m, n) > 0$ since $n(m - 4) > 2$  and $12n - 4 > 0$. For $m = 2$ and $n \geq 2$ we have $f_5(2, n) = 8n - 8 > 0$. For $m = 4$ and $n \geq 2$ we have $f_5(4, n) = 12n - 12 > 0$.
Therefore, for all $m \geq 2$  and $n \geq 2$, we have
\[
\left|\frac{-f_5(m, n)}{2(m + 2)}\right| = \left|\frac{f_5(m, n)}{2(m + 2)}\right| = \frac{m(n(m - 4) - 2) + 12n - 4}{2(m + 2)}.
\]
\[
\left|\frac{n(m - 2)}{2} - \frac{2|e(\mathcal{CCC}(G))|}{|V(\mathcal{CCC}(G))|}\right| = \left|\frac{n(m - 2)}{2} - \frac{n(m - 2)^2 - 2(m - 4n + 2)}{2(m + 2)}\right| =  \left|\frac{f_6(m, n)}{m + 2}\right|,
\]
where $f_6(m, n) = 2n(m - 4) + m + 2$. Clearly, $f_6(m, n) > 0$ if $m \geq 4$ and $n \geq 2$. For $m = 2$ and $n \geq 2$ we have $f_6(2, n) = -4n + 4 < 0$. Therefore,
\[
\left|\frac{f_6(m, n)}{m + 2}\right| =\begin{cases}
n - 1, & \text{ if $m = 2$ and $n \geq 2$}\\
\frac{2n(m - 4) + m + 2}{m + 2}, & \text{ if $m \geq 4$ and $n \geq 2$}.
\end{cases}
\]
\[
\left|n - \frac{2|e(\mathcal{CCC}(G))|}{|V(\mathcal{CCC}(G))|}\right| = \left|n - \frac{n(m - 2)^2 - 2(m - 4n + 2)}{2(m + 2)}\right| = \left|\frac{-f_7(m, n)}{2(m + 2)}\right|,
\]
where $f_7(m, n) = mn(m - 6) - 2m + 8n - 4$. For $m = 2$ and $n \geq 2$ we have $f_7(2, n) =  - 8$. For $m = 4$ and $n \geq 2$ we have $f_7(4, n) =   -12$. For $m = 6$ and $n \geq 2$ we have $f_7(6, n) = 8n - 16 \geq 0$. Also, for $m \geq 8$ and $n \geq 2$ we have  $m^2 \geq 8m$  which gives $m(m - 6) \geq 2m$ and so $mn(m - 6) \geq 2mn > 2m$. Therefore, $mn(m - 6) - 2m > 0$ and so  $f_7(m, n) > 0$ since $8n - 4 > 0$. Hence,
\[
\left|\frac{-f_7(m, n)}{2(m + 2)}\right| = \begin{cases}
1, & \text{ if $m = 2, 4$ and $n \geq 2$}\\
\frac{mn(m - 6) - 2m + 8n - 4}{2(m + 2)}, & \text{ if $m \geq 6$ and $n \geq 2$}.
\end{cases}
\]

\noindent Now, by \eqref{L-energy},  we have
\begin{align*}
LE(\mathcal{CCC}(G)) & = 3 \times \frac{m(n(m - 4) - 2) + 12n - 4}{2(m + 2)} + \frac{n(m - 2) - 2}{2} \times (n - 1) + (2n - 2) \times 1\\
& = 4(n - 1),
\end{align*}
if $m = 2$ and $n \geq 2$. If $m = 4$ and $n \geq 2$ then
\begin{align*}
LE(\mathcal{CCC}(G)) & = 3 \times \frac{m(n(m - 4) - 2) + 12n - 4}{2(m + 2)} + \frac{n(m - 2) - 2}{2} \times \frac{2n(m - 4) + m + 2}{m + 2}\\
& \hspace{5.3cm} + (2n - 2) \times 1\\
& = 6(n - 1).
\end{align*}
If  $m \geq 6$ and $n \geq 2$ then
\begin{align*}
LE(\mathcal{CCC}(G)) & = 3 \times \frac{m(n(m - 4) - 2) + 12n - 4}{2(m + 2)} + \frac{n(m - 2) - 2}{2} \times \frac{2n(m - 4) + m + 2}{m + 2}\\
& \hspace{5.3cm} + (2n - 2) \times \frac{mn(m - 6) - 2m + 8n - 4}{2(m + 2)}\\
& = \frac{2m^2n^2 - 12mn^2 + m^2n + 16n^2 - 4mn -  2m + 12n - 4}{m + 2}.
\end{align*}

Again,
\[
\left|n(m - 2) - 2 - \frac{2|e(\mathcal{CCC}(G))|}{|V(\mathcal{CCC}(G))|}\right| = \left|n(m - 2) - 2 - \frac{n(m - 2)^2 - 2(m - 4n + 2)}{2(m + 2)}\right| = \left|\frac{f_8(m, n)}{2(m + 2)}\right|,
\]
where $f_8(m, n) = n(m^2 -20) + 2m(n - 1) + 2mn - 4$. For $m = 2$ and $n \geq 2$ we have $f_8(2, n) = -8n - 8 < 0$.  For $m = 4$ and $n \geq 2$ we have $f_8(4, n) = 12n - 12 > 0$.  For $m \geq 6$ and $n \geq 2$ we have $f_8(m, n) > 0$. Therefore,
\[
\left|\frac{f_8(m, n)}{2(m + 2)}\right| = \begin{cases}
n + 1, & \text{ if $m = 2$ and $n \geq 2$}\\
n - 1, & \text{ if $m = 4$ and $n \geq 2$}\\
\frac{n(m^2 -20) + 2m(n - 1) + 2mn - 4}{2(m + 2)}, & \text{ if $m \geq 6$ and $n \geq 2$.}
\end{cases}
\]
\[
\left|\frac{n(m - 2) - 4}{2} - \frac{2|e(\mathcal{CCC}(G))|}{|V(\mathcal{CCC}(G))|}\right| = \left|\frac{n(m - 2) - 4}{2} - \frac{n(m - 2)^2 - 2(m - 4n + 2)}{2(m + 2)}\right| = \left|\frac{f_9(m, n)}{m + 2}\right|,
\]
where $f_9(m, n) = n(m - 8) + m(n - 1) - 2$. For $m = 2$ and $n \geq 2$ we have $f_9(2, n) = -4n - 4 < 0$. For $m = 4$ and $n \geq 2$ we have $f_9(4, n) = - 6$. For $m = 6$ and $n \geq 2$ we have $f_9(6, n) = 4n - 8 \geq 0$. 
Further, if For $m \geq 8$ and $n \geq 2$ then $f_9(m, n) > 0$ since  $n(m - 8) \geq 0$ and $m(n - 1) - 2 > 0$. Hence,
\[
\left|\frac{f_9(m, n)}{m + 2}\right| = \begin{cases}
n + 1, & \text{ if $m = 2$ and $n \geq 2$}\\
1, & \text{ if $m = 4$ and $n \geq 2$}\\
\frac{n(m - 8) + m(n - 1) - 2}{m + 2}, & \text{ if $m \geq 6$ and $n \geq 2$.}
\end{cases}
\]
\[
\left|2n - 2 - \frac{2|e(\mathcal{CCC}(G))|}{|V(\mathcal{CCC}(G))|}\right| = \left|2n - 2 - \frac{n(m - 2)^2 - 2(m - 4n + 2)}{2(m + 2)}\right| =  \left|\frac{-f_{10}(m, n)}{2(m + 2)}\right|,
\]
where $f_{10}(m, n) = n(m^2 - 8m + 4) + 2m + 4$. Clearly, $f_{10}(m, n) > 0$ for $m \geq 8$ and $n \geq 2$. For $m = 2$ and $n \geq 2$ we have  $f_{10}(2, n) = -8n + 8 < 0$.  For $m = 4$ and $n \geq 2$ we have  $f_{10}(4, n) = -12n + 12 < 0$. For $m = 6$ and $n \geq 2$ we have  $f_{10}(6, n) = -8n + 16 \leq 0$. 
Hence,
\[
\left|\frac{f_{10}(m, n)}{m + 2}\right| = \begin{cases}
n - 1, & \text{ if $m = 2$ and $n \geq 2$}\\
n - 1, & \text{ if $m = 4$ and $n \geq 2$}\\
\frac{n - 2}{2}, & \text{ if $m = 6$ and $n \geq 2$}\\
\frac{n(m^2 - 8m + 4) + 2m + 4}{2(m + 2)}, & \text{ if $m \geq 8$ and $n \geq 2$.} 
\end{cases}
\]
\[
\left|n - 2 - \frac{2|e(\mathcal{CCC}(G))|}{|V(\mathcal{CCC}(G))|}\right| = \left|n - 2 - \frac{n(m - 2)^2 - 2(m - 4n + 2)}{2(m + 2)}\right| = \left|\frac{-f_{11}(m, n)}{2(m + 2)}\right|,
\]
where $f_{11}(m, n) = n(m - 2)(m - 4) + 2m + 4$. Note that for $m \geq 4$ and $n \geq 2$ we have $f_{11}(m, n) > 0$. For $m = 2$ and $n \geq 2$ we have $f_{11}(m, n) = 8$. Therefore,
\[
\left|\frac{-f_{11}(m, n)}{2(m + 2)}\right| = \frac{f_{11}(m, n)}{2(m + 2)} = \frac{n(m - 2)(m - 4) + 2m + 4}{2(m + 2)}.
\]

\noindent By \eqref{Q-energy},  we have
\begin{align*}
LE^+(\mathcal{CCC}(G)) & = n + 1 + \frac{n(m - 2) - 2}{2} \times (n + 1) + 2 \times (n - 1)\\
& \hspace{4cm} + (2n - 2) \times \frac{n(m - 2)(m - 4) + 2m + 4}{2(m + 2)} \\
& = 4(n - 1),
\end{align*}
if $m = 2$ and $n \geq 2$.  If $m = 4$ and $n \geq 2$ then
\begin{align*}
LE^+(\mathcal{CCC}(G)) & = n - 1 + \frac{n(m - 2) - 2}{2} \times 1 + 2 \times (n - 1) + (2n - 2) \times \frac{n(m - 2)(m - 4) + 2m + 4}{2(m + 2)} \\
& = 6(n - 1).
\end{align*}
If $m = 6$ and $n \geq 2$ then
\begin{align*}
LE^+(\mathcal{CCC}(G)) & = \frac{n(m^2 -20) + 2m(n - 1) + 2mn - 4}{2(m + 2)} + \frac{n(m - 2) - 2}{2} \times\frac{n(m - 8) + m(n - 1) - 2}{m + 2}  \\
& \hspace{4cm} + 2 \times \frac{n - 2}{2} + (2n - 2) \times \frac{n(m - 2)(m - 4) + 2m + 4}{2(m + 2)}\\
& = 2(n + 2)(n - 1).
\end{align*}
If $m \geq 8$ and $n \geq 2$ then
\begin{align*}
LE^+(\mathcal{CCC}(G)) & = \frac{n(m^2 -20) + 2m(n - 1) + 2mn - 4}{2(m + 2)} + \frac{n(m - 2) - 2}{2} \times \frac{n(m - 8) + m(n - 1) - 2}{m + 2}\\
& \hspace{1cm} + 2 \times \frac{n(m^2 - 8m + 4) + 2m + 4}{2(m + 2)} + (2n - 2) \times \frac{n(m - 2)(m - 4) + 2m + 4}{2(m + 2)} \\
& = \frac{2n^2(m - 2)(m - 4)}{m + 2}.
\end{align*}
This completes the proof.
\end{proof}
\begin{theorem}\label{V(8n)}
If $G = V_{8n}$ then
\begin{enumerate}
\item $\spec(\mathcal{CCC}(G)) = \begin{cases} 
\left\{(-1)^{2n - 2}, 0^{2},  \left(2n - 2\right)^{1}\right\}, & \text{ if $n$  is odd}\\
\left\{(-1)^{2n - 1}, 1^{2},  \left(2n - 3\right)^{1}\right\}, & \text{ if $n$  is  even} 
\end{cases}$ 

and $E(\mathcal{CCC}(G))= \begin{cases} 
4n - 4, & \text{ if $n$  is odd}\\
4n - 2, & \text{ if $n$  is  even.} 
\end{cases}$
\item $\L-spec(\mathcal{CCC}(G)) = \begin{cases} 
\left\{0^{3}, \left(2n - 1\right)^{2n - 2}\right\}, & \text{ if $n$  is odd}\\
\left\{0^{3}, 2^2, \left(2n - 2\right)^{2n - 3}\right\}, & \text{ if $n$  is  even} 
\end{cases}$ 

and $LE(\mathcal{CCC}(G))= \begin{cases} 
\frac{6(2n - 1)(2n - 2)}{2n + 1}, & \text{ if $n$  is odd}\\
6, & \text{ if $n = 2$}\\
\frac{2(2n - 3)(5n - 7)}{n + 1}, & \text{ if $n$  is  even and $n \geq 4$.} 
\end{cases}$ 
\item $\Q-spec(\mathcal{CCC}(G)) = \begin{cases} 
\left\{0^{2}, (4n - 4)^{1}, \left(2n - 3\right)^{2n - 2}\right\}, & \text{ if $n$  is odd}\\
\left\{2^2, 0^{2}, (4n - 6)^{1}, \left(2n - 4\right)^{2n - 3}\right\}, & \text{ if $n$  is  even} 
\end{cases}$ 

and $LE^+(\mathcal{CCC}(G))= \begin{cases} 
\frac{4(2n - 1)(2n - 2)}{2n + 1}, & \text{ if $n$  is odd}\\
6, & \text{ if $n = 2$}\\
\frac{16(n - 1)(n - 2)}{n + 1}, & \text{ if $n$  is  even and $n \geq 4$.} 
\end{cases}$
\end{enumerate}
\end{theorem}
\begin{proof}
We shall prove the result by considering the following cases.

\noindent \textbf{Case 1.} $n$ is odd.

 By \cite[Proposition 2.4]{sA2020} we have $\mathcal{CCC}(G) = 2K_1 \sqcup K_{2n - 1}$.  Therefore, by Theorem \ref{prethm1}, it follows that
\begin{align*}
&\spec(\mathcal{CCC}(G)) = \left\{(-1)^{2n - 2}, 0^{2},  \left(2n - 2\right)^{1}\right\}, \quad  \L-spec(\mathcal{CCC}(G)) = \left\{0^{3}, \left(2n - 1\right)^{2n - 2}\right\} \\
\text{ and } &\Q-spec(\mathcal{CCC}(G)) = \left\{0^{2}, (4n - 4)^{1}, \left(2n - 3\right)^{2n - 2}\right\}.
\end{align*}
Hence, by \eqref{energy}, we get
\[
E(\mathcal{CCC}(G))= 2n - 2 + 2n - 2 = 4n - 4.
\]

We have $|V(\mathcal{CCC}(G))| = 2n + 1$ and $|e(\mathcal{CCC}(G))| = \frac{(2n - 1)(2n - 2)}{2}$. Therefore, $\frac{2|e(\mathcal{CCC}(G))|}{|V(\mathcal{CCC}(G))|} = \frac{(2n - 1)(2n - 2)}{2n + 1}$.
Also, 
\[
\left|0 - \frac{2|e(\mathcal{CCC}(G))|}{|V(\mathcal{CCC}(G))|}\right| = \left|0 - \frac{(2n - 1)(2n - 2)}{2n + 1}\right| = \frac{(2n - 1)(2n - 2)}{2n + 1} \quad \text{ and}
\]
\[
\left|2n - 1 - \frac{2|e(\mathcal{CCC}(G))|}{|V(\mathcal{CCC}(G))|}\right| = \left|2n - 1 - \frac{(2n - 1)(2n - 2)}{2n + 1}\right| = \frac{3(2n - 1)}{2n + 1}.
\]

\noindent Now, by \eqref{L-energy},  we have
\[
LE(\mathcal{CCC}(G))= 3 \times \frac{(2n - 1)(2n - 2)}{2n + 1} + (2n - 2)\times \frac{3(2n - 1)}{2n + 1} = \frac{6(2n - 1)(2n - 2)}{2n + 1}.
\] 

Again,
\[
\left|4n - 4 - \frac{2|e(\mathcal{CCC}(G))|}{|V(\mathcal{CCC}(G))|}\right| = \left|4n - 4 - \frac{(2n - 1)(2n - 2)}{2n + 1}\right| = \frac{(2n - 2)(2n + 3)}{2n + 1} \quad \text{ and}
\]
\[
\left|2n - 3 - \frac{2|e(\mathcal{CCC}(G))|}{|V(\mathcal{CCC}(G))|}\right| = \left|2n - 3 - \frac{(2n - 1)(2n - 2)}{2n + 1}\right| = \frac{2n - 5}{2n + 1}.
\]

\noindent By \eqref{Q-energy},  we have
\begin{align*}
LE^+(\mathcal{CCC}(G)) &= 2\times \frac{(2n - 1)(2n - 2)}{2n + 1} + \frac{(2n - 2)(2n + 3)}{2n + 1} + (2n - 2) \times \frac{2n - 5}{2n + 1} \\
&= \frac{4(2n - 1)(2n - 2)}{2n + 1}.
\end{align*} 

\noindent \textbf{Case 2.} $n$ is even.

 By \cite[Proposition 2.4]{sA2020} we have $\mathcal{CCC}(G) = 2K_2 \sqcup K_{2n - 2}$.  Therefore, by Theorem \ref{prethm1}, it follows that
\begin{align*}
&\spec(\mathcal{CCC}(G)) = \left\{(-1)^{2n - 1}, 1^{2},  \left(2n - 3\right)^{1}\right\}, \quad  \L-spec(\mathcal{CCC}(G)) = \left\{0^{3}, 2^2, \left(2n - 2\right)^{2n - 3}\right\} \\
\text{ and } &\Q-spec(\mathcal{CCC}(G)) = \left\{2^2, 0^{2}, (4n - 6)^{1}, \left(2n - 4\right)^{2n - 3}\right\}.
\end{align*}
Hence, by \eqref{energy}, we get 
\[
E(\mathcal{CCC}(G))= 2n - 1 + 2 + 2n - 3 = 4n - 2.
\]

We have $|V(\mathcal{CCC}(G))| = 2n + 2$ and $|e(\mathcal{CCC}(G))| = \frac{(2n - 2)(2n - 3) + 4}{2}$. Therefore, $\frac{2|e(\mathcal{CCC}(G))|}{|V(\mathcal{CCC}(G))|} = \frac{(n - 1)(2n - 3) + 2}{n + 1}$.
Also, 
\[
\left|0 - \frac{2|e(\mathcal{CCC}(G))|}{|V(\mathcal{CCC}(G))|}\right| = \left|0 - \frac{(n - 1)(2n - 3) + 2}{n + 1}\right| = \frac{(n - 1)(2n - 3) + 2}{n + 1},
\]
\[
\left|2 - \frac{2|e(\mathcal{CCC}(G))|}{|V(\mathcal{CCC}(G))|}\right| = \left|2 - \frac{(n - 1)(2n - 3) + 2}{n + 1}\right| = \left|\frac{-(2n - 1)(n - 3)}{n + 1}\right| = \begin{cases}
1, & \text{ if } n = 2\\
\frac{(2n - 1)(n - 3)}{n + 1}, & \text{ if } n \geq 4
\end{cases}   
\]
and
\[
  \left|2n - 2 - \frac{2|e(\mathcal{CCC}(G))|}{|V(\mathcal{CCC}(G))|}\right| = \left|2n - 2 - \frac{(n - 1)(2n - 3) + 2}{n + 1}\right| = \frac{5n - 7}{n + 1}.
\]

\noindent Now, by \eqref{L-energy},  we have
\[
LE(\mathcal{CCC}(G))= 3\times \frac{(n - 1)(2n - 3) + 2}{n + 1} + 2 \times 1 + (2n - 3) \times \frac{5n - 7}{n + 1} = 6,
\] 
if $n = 2$. If $n \geq 4$ then
\begin{align*}
LE(\mathcal{CCC}(G)) & = 3\times \frac{(n - 1)(2n - 3) + 2}{n + 1} + 2 \times \frac{(2n - 1)(n - 3)}{n + 1} + (2n - 3) \times \frac{5n - 7}{n + 1}\\
& = \frac{2(10n^2 - 29n + 21)}{n + 1} = \frac{2(2n - 3)(5n - 7)}{n + 1}.
\end{align*} 

Again,
\[
\left|4n - 6 - \frac{2|e(\mathcal{CCC}(G))|}{|V(\mathcal{CCC}(G))|}\right| = \left|4n - 6 - \frac{(n - 1)(2n - 3) + 2}{n + 1}\right| = \frac{2n^2 + 3n - 11}{n + 1} \quad \text{ and } 
\]
\[
\left|2n - 4 - \frac{2|e(\mathcal{CCC}(G))|}{|V(\mathcal{CCC}(G))|}\right| = \left|2n - 4 - \frac{(n - 1)(2n - 3) + 2}{n + 1}\right| = \left|\frac{3n - 9}{n + 1}\right| = \begin{cases}
1, & \text{ if } n = 2\\
\frac{3n - 9}{n + 1}, & \text{ if } n \geq 4.
\end{cases}
\]

\noindent By \eqref{Q-energy},  we have
\begin{align*}
LE^+(\mathcal{CCC}(G)) & = 2 \times 1 + 2  \times \frac{(n - 1)(2n - 3) + 2}{n + 1} + \frac{2n^2 + 3n - 11}{n + 1} + (2n - 3) \times 1 = 6,
\end{align*}
if $n = 2$. If $n \geq 4$ then
\begin{align*}
LE^+&(\mathcal{CCC}(G))\\
& = 2 \times \frac{(2n - 1)(n - 3)}{n + 1} + 2  \times \frac{(n - 1)(2n - 3) + 2}{n + 1} + \frac{2n^2 + 3n - 11}{n + 1} + (2n - 3) \times \frac{3n - 9}{n + 1}\\
& = \frac{16(n - 1)(n - 2)}{n + 1}.
\end{align*}
This completes the proof.
\end{proof}

\begin{theorem}\label{SD(8n)}
If $G = SD_{8n}$ then
\begin{enumerate}
\item $\spec(\mathcal{CCC}(G)) = \begin{cases} 
\left\{(-1)^{2n}, 3^{1},  \left(2n - 3\right)^{1}\right\}, & \text{ if $n$  is odd}\\
\left\{(-1)^{2n - 2}, 0^{2},  \left(2n - 2\right)^{1}\right\}, & \text{ if $n$  is  even} 
\end{cases}$ 

and $E(\mathcal{CCC}(G))= \begin{cases} 
4n, & \text{ if $n$  is odd}\\
4n - 4, & \text{ if $n$  is  even.} 
\end{cases}$
\item $\L-spec(\mathcal{CCC}(G)) = \begin{cases} 
\left\{0^{2}, 4^3,  \left(2n - 2\right)^{2n - 3}\right\}, & \text{ if $n$  is odd}\\
\left\{0^{3},   \left(2n - 1\right)^{2n - 2}\right\}, & \text{ if $n$  is  even} 
\end{cases}$ 

and $LE(\mathcal{CCC}(G))= \begin{cases} 
12, & \text{ if $n = 3$}\\
\frac{2(2n - 3)(5n - 11)}{n + 1}, & \text{ if $n$  is odd and $n \geq 5$}\\
\frac{6(2n - 1)(2n - 2)}{2n + 1}, & \text{ if $n$  is  even.} 
\end{cases}$
\item $\Q-spec(\mathcal{CCC}(G)) = \begin{cases} 
\left\{6^1, 2^{3}, (4n - 6)^{1}, \left(2n - 4\right)^{2n - 3}\right\}, & \text{ if $n$  is odd}\\
\left\{0^{2}, (4n - 4)^{1}, \left(2n - 3\right)^{2n - 2}\right\}, & \text{ if $n$  is  even} 
\end{cases}$ 

and $LE^+(\mathcal{CCC}(G))= \begin{cases} 
12, & \text{ if $n = 3$}\\
22, & \text{ if $n = 5$}\\
\frac{16(n - 1)(n - 3)}{n + 1}, & \text{ if $n$  is odd and $n \geq 7$}\\
\frac{28}{5}, & \text{ if $n = 2$}\\
\frac{4(2n - 1)(2n - 2)}{2n + 1}, & \text{ if $n$  is  even and $n \geq 4$.} 
\end{cases}$
\end{enumerate}
\end{theorem}

\begin{proof}
We shall prove the result by considering the following cases.

\noindent \textbf{Case 1.} $n$ is odd.

 By \cite[Proposition 2.5]{sA2020} we have $\mathcal{CCC}(G) = K_4 \sqcup K_{2n - 2}$.  Therefore, by Theorem \ref{prethm1}, it follows that
\begin{align*}
&\spec(\mathcal{CCC}(G)) = \left\{(-1)^{2n}, 3^{1},  \left(2n - 3\right)^{1}\right\}, \quad  \L-spec(\mathcal{CCC}(G)) = \left\{0^{2}, 4^3,  \left(2n - 2\right)^{2n - 3}\right\} \\
\text{ and } &\Q-spec(\mathcal{CCC}(G)) = \left\{6^1, 2^{3}, (4n - 6)^{1}, \left(2n - 4\right)^{2n - 3}\right\}.
\end{align*}
Hence, by \eqref{energy}, we get
\[
E(\mathcal{CCC}(G))= 2n + 3 + 2n - 3 = 4n.
\]

We have $|V(\mathcal{CCC}(G))| = 2n + 2$ and $|e(\mathcal{CCC}(G))| = \frac{(2n - 2)(2n - 3) +12}{2}$. Therefore, $\frac{2|e(\mathcal{CCC}(G))|}{|V(\mathcal{CCC}(G))|} = \frac{(n - 1)(2n - 3) + 6}{n + 1}$.
Also, 
\[
\left|0 - \frac{2|e(\mathcal{CCC}(G))|}{|V(\mathcal{CCC}(G))|}\right| = \left|0 - \frac{(n - 1)(2n - 3) + 6}{n + 1}\right| = \frac{(n - 1)(2n - 3) + 6}{n + 1},
\]
\[
\left|4 - \frac{2|e(\mathcal{CCC}(G))|}{|V(\mathcal{CCC}(G))|}\right| = \left|4 - \frac{(n - 1)(2n - 3) + 6}{n + 1}\right| = \left|\frac{-2n^2 + 9n - 5}{n + 1}\right| = \begin{cases}
1, & \text{ if } n = 3\\
\frac{2n^2 - 9n + 5}{n + 1}, & \text{ if } n \geq 5
\end{cases}
\]
and
\[
\left|2n - 2 - \frac{2|e(\mathcal{CCC}(G))|}{|V(\mathcal{CCC}(G))|}\right| = \left|2n - 2 - \frac{(n - 1)(2n - 3) + 6}{n + 1}\right| = \frac{5n - 11}{n + 1}.
\]

\noindent Now, by \eqref{L-energy},  we have
\[
LE(\mathcal{CCC}(G))= 2\times \frac{(n - 1)(2n - 3) + 6}{n + 1} + 3\times 1 + (2n - 3)\times \frac{5n - 11}{n + 1} = 12,
\] 
if $n = 3$. If $n \geq 5$ then
\begin{align*}
LE(\mathcal{CCC}(G)) & = 2\times \frac{(n - 1)(2n - 3) + 6}{n + 1} + 3\times \frac{2n^2 - 9n + 5}{n + 1} + (2n - 3)\times \frac{5n - 11}{n + 1} \\
& = \frac{2(10n^2 - 37n + 33)}{n + 1} = \frac{2(2n - 3)(5n - 11)}{n + 1}.
\end{align*} 

Again,
\[
\left|6 - \frac{2|e(\mathcal{CCC}(G))|}{|V(\mathcal{CCC}(G))|}\right| = \left|6 - \frac{(n - 1)(2n - 3) + 6}{n + 1}\right| = \left|\frac{- 2n^2 + 11n - 3}{n + 1}\right| = \begin{cases}
\frac{- 2n^2 + 11n - 3}{n + 1}, & \text{ if } n = 3, 5\\
\frac{2n^2 - 11n + 3}{n + 1}, & \text{ if } n \geq 7,
\end{cases}
\]
\[
\left|2 - \frac{2|e(\mathcal{CCC}(G))|}{|V(\mathcal{CCC}(G))|}\right| = \left|2 - \frac{(n - 1)(2n - 3) + 6}{n + 1}\right| = \frac{2n^2 - 7n + 7}{n + 1},
\]
\[
\left|4n - 6 - \frac{2|e(\mathcal{CCC}(G))|}{|V(\mathcal{CCC}(G))|}\right| = \left|4n - 6 - \frac{(n - 1)(2n - 3) + 6}{n + 1}\right| = \frac{2n^2 + 3n - 15}{n + 1} \quad \text{ and }
\]
\[
\left|2n - 4 - \frac{2|e(\mathcal{CCC}(G))|}{|V(\mathcal{CCC}(G))|}\right| = \left|2n - 4 - \frac{(n - 1)(2n - 3) + 6}{n + 1}\right| = \left|\frac{3n - 13}{n + 1}\right| = \begin{cases}
1, & \text{ if } n =3\\
\frac{3n - 13}{n + 1}, & \text{ if } n \geq 5.
\end{cases}
\]

\noindent By \eqref{Q-energy},  we have
\begin{align*}
LE^+(\mathcal{CCC}(G)) &=  \frac{- 2n^2 + 11n - 3}{n + 1} + 3 \times \frac{2n^2 - 7n + 7}{n + 1} + \frac{2n^2 + 3n - 15}{n + 1} + (2n - 3) \times 1  = 12,
\end{align*} 
if $n = 3$. If $n =5$ then
\begin{align*}
LE^+(\mathcal{CCC}(G)) &=  \frac{- 2n^2 + 11n - 3}{n + 1} + 3 \times \frac{2n^2 - 7n + 7}{n + 1} +  \frac{2n^2 + 3n - 15}{n + 1} + (2n - 3) \times \frac{3n - 13}{n + 1} \\
&= 22.
\end{align*}
If $n \ge 7$ then
\begin{align*}
LE^+(\mathcal{CCC}(G)) &=  \frac{2n^2 - 11n + 3}{n + 1} + 3 \times \frac{2n^2 - 7n + 7}{n + 1} +  \frac{2n^2 + 3n - 15}{n + 1} + (2n - 3) \times \frac{3n - 13}{n + 1} \\
&= \frac{16(n - 1)(n - 3)}{n + 1}. 
\end{align*}

\noindent \textbf{Case 2.} $n$ is even.

 By \cite[Proposition 2.5]{sA2020} we have $\mathcal{CCC}(G) = 2K_1 \sqcup K_{2n - 1}$.  Therefore, by Theorem \ref{prethm1}, it follows that
\begin{align*}
&\spec(\mathcal{CCC}(G)) = \left\{(-1)^{2n - 2}, 0^{2},  \left(2n - 2\right)^{1}\right\}, \quad  \L-spec(\mathcal{CCC}(G)) = \left\{0^{3},   \left(2n - 1\right)^{2n - 2}\right\} \\
\text{ and } &\Q-spec(\mathcal{CCC}(G)) = \left\{0^{2}, (4n - 4)^{1}, \left(2n - 3\right)^{2n - 2}\right\}.
\end{align*}
Hence, by \eqref{energy}, we get
\[
E(\mathcal{CCC}(G))= 2n - 2 + 2n - 2 = 4n - 4.
\]

We have $V(\mathcal{CCC}(G)) = 2n + 1$ and $e(\mathcal{CCC}(G)) = \frac{(2n - 1)(2n - 2)}{2}$. So, $\frac{2|e(\mathcal{CCC}(G))|}{|V(\mathcal{CCC}(G))|} = \frac{(2n - 1)(2n - 2)}{2n + 1}$.

\noindent Also, 
\[
\left|0 - \frac{2|e(\mathcal{CCC}(G))|}{|V(\mathcal{CCC}(G))|}\right| = \left|0 - \frac{(2n - 1)(2n - 2)}{2n + 1}\right| = \frac{(2n - 1)(2n - 2)}{2n + 1} \quad \text{ and }
\]
\[
\left|2n - 1 - \frac{2|e(\mathcal{CCC}(G))|}{|V(\mathcal{CCC}(G))|}\right| = \left|2n - 1 - \frac{(2n - 1)(2n - 2)}{2n + 1}\right| = \frac{3(2n - 1)}{2n + 1}.
\]

\noindent Now, by \eqref{L-energy},  we have
\[
LE(\mathcal{CCC}(G))= 3 \times \frac{(2n - 1)(2n - 2)}{2n + 1} + (2n - 2) \times \frac{3(2n - 1)}{2n + 1} = \frac{6(2n - 1)(2n - 2)}{2n + 1}.
\] 

Again,
\[
\left|4n - 4 - \frac{2|e(\mathcal{CCC}(G))|}{|V(\mathcal{CCC}(G))|}\right| = \left|4n - 4 - \frac{(2n - 1)(2n - 2)}{2n + 1}\right| = \frac{(2n - 2)(2n + 3)}{2n + 1} \quad \text{ and }
\]
\[
\left|2n - 3 - \frac{2|e(\mathcal{CCC}(G))|}{|V(\mathcal{CCC}(G))|}\right| = \left|2n - 3 - \frac{(2n - 1)(2n - 2)}{2n + 1}\right| = \left|\frac{2n - 5}{2n + 1}\right| = \begin{cases}
\frac{1}{5}, & \text{ if } n = 2\\
\frac{2n - 5}{2n + 1}, & \text{ if } n \geq 4.
\end{cases}
\]

\noindent By \eqref{Q-energy},  we have
\[
LE^+(\mathcal{CCC}(G))= 2 \times \frac{(2n - 1)(2n - 2)}{2n + 1} +  \frac{(2n - 2)(2n + 3)}{2n + 1} + (2n - 2) \times \frac{1}{5} = \frac{28}{5},
\] 
if $n = 2$. If $n \geq 4$ then
\begin{align*}
LE^+(\mathcal{CCC}(G)) & = 2 \times \frac{(2n - 1)(2n - 2)}{2n + 1} +  \frac{(2n - 2)(2n + 3)}{2n + 1} + (2n - 2) \times \frac{2n - 5}{2n + 1}\\
& = \frac{4(2n - 1)(2n - 2)}{2n + 1}.
\end{align*} 
This completes the proof.
\end{proof}

We conclude this section with the following corollary.
\begin{corollary}
If $G$ is isomorphic to $D_{2n}, Q_{4m}, U_{(n, m)}, V_{8n}$ or $SD_{8n}$ then $\mathcal{CCC}(G)$ is super integral.
\end{corollary}
\section{Comparing various energies}
In this section we compare various energies of $\mathcal{CCC}(G)$  obtained in Section 3  and derive the following relations.

%
%
%
%
\begin{theorem}\label{cD(2n)}
Let $G = D_{2n}$.
\begin{enumerate}
\item If $n=3, 4, 6$ then $E(\mathcal{CCC}(G)) = LE^+(\mathcal{CCC}(G)) = LE(\mathcal{CCC}(G))$.
\item If $n=5$ then $E(\mathcal{CCC}(G)) < LE^+(\mathcal{CCC}(G)) = LE(\mathcal{CCC}(G))$.
\item If $n=10$ then $LE^+(\mathcal{CCC}(G)) < E(\mathcal{CCC}(G)) < LE(\mathcal{CCC}(G))$.
\item If $n\geq 7$ but $n \ne 10$ then $ E(\mathcal{CCC}(G)) < LE^+(\mathcal{CCC}(G)) < LE(\mathcal{CCC}(G))$.
\end{enumerate}
\end{theorem}
\begin{proof}
We shall prove the result by considering the following cases.

\noindent \textbf{Case 1.} $n$ is odd.

If $n=3$ then, by Theorem \ref{CCC(G)-D-2n}, we have
\[
E(\mathcal{CCC}(G)) = LE^+(\mathcal{CCC}(G)) = LE(\mathcal{CCC}(G)) = 0.
\]
If $n=5$ then, by Theorem \ref{CCC(G)-D-2n}, we have
\[
E(\mathcal{CCC}(G)) - LE^+(\mathcal{CCC}(G)) = n - 3 - \frac{(n - 3)(n + 3)}{n + 1} = -\frac{4}{5} < 0
\]
and
$LE^+(\mathcal{CCC}(G)) = LE(\mathcal{CCC}(G)) = \frac{8}{3}$.
Therefore, $E(\mathcal{CCC}(G)) < LE^+(\mathcal{CCC}(G)) = LE(\mathcal{CCC}(G))$.

 If $n\geq 7$ then, by Theorem \ref{CCC(G)-D-2n}, we have
\[
E(\mathcal{CCC}(G)) - LE^+(\mathcal{CCC}(G)) = n - 3 - \frac{(n - 3)(n + 3)}{n + 1} = -\frac{2(n - 3)}{n + 1} < 0
\]
and
\[
LE^+(\mathcal{CCC}(G)) - LE(\mathcal{CCC}(G)) = \frac{(n - 3)(n + 3)}{n + 1} - \frac{2(n - 1)(n - 3)}{n + 1} = -\frac{(n - 3)(n - 5)}{n + 1} < 0.
\]
Therefore, $E(\mathcal{CCC}(G)) < LE^+(\mathcal{CCC}(G)) < LE(\mathcal{CCC}(G))$.

\noindent \textbf{Case 2.} $n$ is even.

Consider  the following  subcases.

\noindent \textbf{Subcase 2.1}  $\frac{n}{2}$ is even.

If $n = 4$ then, by Theorem \ref{CCC(G)-D-2n}, we have
\[
E(\mathcal{CCC}(G)) = LE^+(\mathcal{CCC}(G)) = LE(\mathcal{CCC}(G)) = 0.
\]
If $n = 8$ then, by Theorem \ref{CCC(G)-D-2n}, we have
\[
E(\mathcal{CCC}(G)) - LE^+(\mathcal{CCC}(G)) =  n - 4 - \frac{(n - 4)(n + 6)}{n + 2} = -\frac{8}{5} < 0
\]
and
\[
LE^+(\mathcal{CCC}(G)) - LE(\mathcal{CCC}(G)) = \frac{(n - 4)(n + 6)}{n + 2} - \frac{3(n - 2)(n - 4)}{n + 2} = -\frac{8}{5} < 0.
\]
Therefore, $E(\mathcal{CCC}(G)) < LE^+(\mathcal{CCC}(G)) < LE(\mathcal{CCC}(G))$.

If $n \geq 12$ then, by Theorem \ref{CCC(G)-D-2n}, we have
\[
E(\mathcal{CCC}(G)) - LE^+(\mathcal{CCC}(G)) =  n - 4 - \frac{2(n - 2)(n - 4)}{n + 2} = -\frac{(n - 4)(n - 6)}{n + 2} < 0
\]
and

\[
LE^+(\mathcal{CCC}(G)) - LE(\mathcal{CCC}(G)) = \frac{2(n - 2)(n - 4)}{n + 2} - \frac{3(n - 2)(n - 4)}{n + 2} = -\frac{(n - 2)(n - 4)}{n + 2} < 0.
\]
Therefore, $E(\mathcal{CCC}(G)) < LE^+(\mathcal{CCC}(G)) < LE(\mathcal{CCC}(G))$.

\noindent \textbf{Subcase 2.2}  $\frac{n}{2}$ is odd.

If $n = 6$ then, by Theorem \ref{CCC(G)-D-2n}, we have
\[
E(\mathcal{CCC}(G)) = LE^+(\mathcal{CCC}(G)) = LE(\mathcal{CCC}(G)) = 4.
\]
If $n = 10$ then, by Theorem \ref{CCC(G)-D-2n}, we have
\[
LE^+(\mathcal{CCC}(G)) - E(\mathcal{CCC}(G)) = \frac{22}{3} - ( n - 2 ) = -\frac{2}{3} < 0
\]
and
\[
E(\mathcal{CCC}(G)) - LE(\mathcal{CCC}(G)) = n - 2 - \frac{(n - 4)(3n - 10)}{n + 2} = -2 < 0.
\]
Therefore, $LE^+(\mathcal{CCC}(G)) < E(\mathcal{CCC}(G)) < LE(\mathcal{CCC}(G))$.

If $n \geq 14$  then, by Theorem \ref{CCC(G)-D-2n}, we have
\[
E(\mathcal{CCC}(G)) - LE^+(\mathcal{CCC}(G)) = n - 2 - \frac{2(n - 2)(n - 6)}{n + 2} = -\frac{(n - 2)(n - 10)}{n + 2} < 0
\]
and
\begin{align*}
LE^+(\mathcal{CCC}(G)) - LE(\mathcal{CCC}(G)) &= \frac{2(n - 2)(n - 6)}{n + 2} - \frac{(n - 4)(3n - 10)}{n + 2}\\ &= -\frac{n^2 - 6n + 16}{n + 2} = - \frac{n(n - 14) + 8n + 10}{n + 2} < 0.
\end{align*}
Therefore, $ E(\mathcal{CCC}(G)) < LE^+(\mathcal{CCC}(G)) < LE(\mathcal{CCC}(G))$. Hence, the result follows.
\end{proof}

\begin{theorem}\label{cQ(4m)}
Let $G = Q_{4m}$.
\begin{enumerate}
\item If $m=2,3$ then $E(\mathcal{CCC}(G)) = LE^+(\mathcal{CCC}(G)) = LE(\mathcal{CCC}(G))$.
\item If $m=5$ then $LE^+(\mathcal{CCC}(G)) < E(\mathcal{CCC}(G)) < LE(\mathcal{CCC}(G))$.
\item If $m=7$ then $LE^+(\mathcal{CCC}(G)) = E(\mathcal{CCC}(G)) < LE(\mathcal{CCC}(G))$.
\item If $m=4,6$ or $m\geq8$ then $E(\mathcal{CCC}(G)) < LE^+(\mathcal{CCC}(G)) < LE(\mathcal{CCC}(G))$.
\end{enumerate}
\end{theorem}
\begin{proof}
We shall prove the result by considering the following cases.

\noindent \textbf{Case 1.} $m$ is odd. 

If $m=3$ then, by Theorem \ref{Q(4m)}, we have
\[
E(\mathcal{CCC}(G)) = LE^+(\mathcal{CCC}(G)) = LE(\mathcal{CCC}(G)) = 4. 
\]
If $m=5$ then, by Theorem \ref{Q(4m)}, we have
\[
 LE^+(\mathcal{CCC}(G)) - E(\mathcal{CCC}(G)) = \frac{22}{3} - (2m - 2)  = -\frac{2}{3} < 0 
\]
and
\[
E(\mathcal{CCC}(G)) - LE(\mathcal{CCC}(G)) = 2m - 2 - \frac{2(m - 2)(3m - 5)}{m + 1} = -2 < 0.
\]
Therefore, $LE^+(\mathcal{CCC}(G)) < E(\mathcal{CCC}(G)) < LE(\mathcal{CCC}(G))$.

 If $m=7$ then, by Theorem \ref{Q(4m)},  we have
\[
LE^+(\mathcal{CCC}(G)) = E(\mathcal{CCC}(G)) = 12
\]
and
\[
LE^+(\mathcal{CCC}(G)) - LE(\mathcal{CCC}(G)) = \frac{4(m - 1)(m - 3)}{m + 1} - \frac{2(m - 2)(3m - 5)}{m + 1} = -\frac{2(m + 4)(m - 1)}{m + 1} < 0.
\]
Therefore, $LE^+(\mathcal{CCC}(G)) = E(\mathcal{CCC}(G)) < LE(\mathcal{CCC}(G))$.

If $m\geq 9$ then, by Theorem \ref{Q(4m)}, we have
\[
E(\mathcal{CCC}(G)) - LE^+(\mathcal{CCC}(G)) = 2m - 2 - \frac{4(m - 1)(m - 3)}{m + 1} = -\frac{2(m - 1)(m - 7)}{m + 1} < 0
\]
and
\[
LE^+(\mathcal{CCC}(G)) - LE(\mathcal{CCC}(G)) = \frac{4(m - 1)(m - 3)}{m + 1} - \frac{2(m - 2)(3m - 5)}{m + 1} = -\frac{2(m + 4)(m - 1)}{m + 1} < 0.
\]
Therefore, $E(\mathcal{CCC}(G)) < LE^+(\mathcal{CCC}(G)) < LE(\mathcal{CCC}(G))$.

\noindent \textbf{Case 2.} $m$ is even.

If $m = 2$ then, by Theorem \ref{Q(4m)}, we have
\[
E(\mathcal{CCC}(G)) = LE^+(\mathcal{CCC}(G)) = LE(\mathcal{CCC}(G)) = 0. 
\]
If $m = 4$ then, by Theorem \ref{Q(4m)}, we have
\[
E(\mathcal{CCC}(G)) - LE^+(\mathcal{CCC}(G)) = 2m - 4 - \frac{2(m - 2)(m + 3)}{m + 1} = -\frac{8}{5} < 0
\]
and
\[
LE^+(\mathcal{CCC}(G)) - LE(\mathcal{CCC}(G)) = \frac{2(m - 2)(m + 3)}{m + 1} - \frac{6(m - 1)(m - 2)}{m + 1} = -\frac{8}{5} < 0.
\]
Therefore, $E(\mathcal{CCC}(G)) < LE^+(\mathcal{CCC}(G)) < LE(\mathcal{CCC}(G))$.

If $m \geq 6$  then, by Theorem \ref{Q(4m)}, we have
\[
E(\mathcal{CCC}(G)) - LE^+(\mathcal{CCC}(G)) = 2m - 4 - \frac{4(m - 1)(m - 2)}{m + 1} = -\frac{2(m - 2)(m - 3)}{m + 1} < 0
\]
and
\[
LE^+(\mathcal{CCC}(G)) - LE(\mathcal{CCC}(G)) = \frac{4(m - 1)(m - 2)}{m + 1} - \frac{6(m - 1)(m - 2)}{m + 1} = -\frac{2(m - 1)(m - 2)}{m + 1}.
\]
Therefore,  $E(\mathcal{CCC}(G)) < LE^+(\mathcal{CCC}(G)) < LE(\mathcal{CCC}(G))$. Hence, the result follows.
\end{proof}
\begin{theorem}\label{cU(n,m)}
Let $G = U_{(n, m)}$.
\begin{enumerate}
\item If $m=2,3,4$ and $n\geq 2$ then $LE^+(\mathcal{CCC}(G)) = E(\mathcal{CCC}(G)) = LE(\mathcal{CCC}(G)).$
\item If $m=5$ and $n=2,3$; or $m=6$ and $n=2$ then 
\[
LE^+(\mathcal{CCC}(G)) < E(\mathcal{CCC}(G)) < LE(\mathcal{CCC}(G)).
\]
\item If $m=5$ and $n\geq 4$; $m\geq6$ and $n\geq 3$; or $m\geq 8$ and $n\geq 2$ then 
\[
  E(\mathcal{CCC}(G)) < LE^+(\mathcal{CCC}(G)) < LE(\mathcal{CCC}(G)).
\]
\item If $m = 7$ and $n = 2$ then $E(\mathcal{CCC}(G)) = LE^+(\mathcal{CCC}(G)) < LE(\mathcal{CCC}(G)).$
\end{enumerate}
\end{theorem}
\begin{proof}
We shall prove the result by considering the following cases.

\noindent \textbf{Case 1.} If $m$ is odd and $n \geq 2$.

If $m=3$ and $n \geq 2$ then, by Theorem \ref{U(n,m)}, we have
\[
LE^+(\mathcal{CCC}(G)) = E(\mathcal{CCC}(G))= LE(\mathcal{CCC}(G))= 4(n-1).
\]
If $m=5$ and $n=2$ then, by Theorem \ref{U(n,m)}, we have
\[
LE^+(\mathcal{CCC}(G)) - E(\mathcal{CCC}(G))=\frac{2n^2+10n-6}{3}-(n(m+1)-4)=-\frac{2}{3} < 0
\]
and
\[
E(\mathcal{CCC}(G)) - LE(\mathcal{CCC}(G))= n(m+1)-4- \frac{2(2n-1)(n+3)}{3}=-2 < 0.
\]
Therefore, $ LE^+(\mathcal{CCC}(G)) < E(\mathcal{CCC}(G)) < LE(\mathcal{CCC}(G)).$

If $m=5$ and $n= 3$ then, by Theorem \ref{U(n,m)}, we have 
\[
LE^+(\mathcal{CCC}(G)) - E(\mathcal{CCC}(G))=\frac{2(2n+3)(n-1)}{3}-(n(m+1)-4)=-2 < 0
\]
and
\[
E(\mathcal{CCC}(G)) - LE(\mathcal{CCC}(G))= n(m+1)-4 - \frac{2(2n-1)(n+3)}{3}=-4 < 0.
\]
Therefore, $ LE^+(\mathcal{CCC}(G)) < E(\mathcal{CCC}(G)) < LE(\mathcal{CCC}(G)).$

If $m=5$ and $n\geq 4$ then, by Theorem \ref{U(n,m)}, we have 
\begin{align*}
E(\mathcal{CCC}(G)) - LE^+(\mathcal{CCC}(G)) &= n(m+1)-4 - \frac{2(2n+3)(n-1)}{3} \\
&= \frac{-2(2n^2-8n+3)}{3} = \frac{-2(2n(n-4)+3)}{3} < 0
\end{align*}
and
\[
LE^{+}(\mathcal{CCC}(G)) - LE(\mathcal{CCC}(G))= \frac{2(2n + 3)(n - 1)}{3} - \frac{2(2n - 1)(n + 3)}{3}=\frac{-8n}{3} < 0.
\]
Therefore, $E(\mathcal{CCC}(G)) < LE^+(\mathcal{CCC}(G)) < LE(\mathcal{CCC}(G)).$

If $m\geq 7$ and $n\geq 2$ then, by Theorem \ref{U(n,m)}, we have 
\[
E(\mathcal{CCC}(G)) - LE^+(\mathcal{CCC}(G)) = n(m + 1) - 4  - \frac{n^2(m - 1)(m - 3)}{m + 1} = -\frac{f_{1}( m , n )}{m + 1},
\]
where $f_{1}(m , n) = mn(m - 4)(n - 3) + 2mn(m - 7) + 3n(n - 1) + 4(m + 1)$.
For $m\geq 7$ and $n = 2$ we have $f_{1}( m , n )= 2(m-1)(m-7)\geq 0$. Hence, $f_{1}(7, 2) = 0$ and $f_{1}(m, 2) > 0$ if $m \geq 9$. Thus, $E(\mathcal{CCC}(G)) = LE^+(\mathcal{CCC}(G))$ and $E(\mathcal{CCC}(G)) < LE^+(\mathcal{CCC}(G))$ according as if $m = 7$,  $n =2$   and  $m \geq 9$, $n = 2$.
For $m\geq 7$ and $n \geq 3$ we have $f_{1}( m , n ) > 0$ and so $E(\mathcal{CCC}(G)) < LE^+(\mathcal{CCC}(G))$.

If $m\geq 7$ and $n\geq 2$ then, by Theorem \ref{U(n,m)}, we also have 
\begin{align*}
LE^{+}(\mathcal{CCC}(G)) &- LE(\mathcal{CCC}(G))\\ 
& = \frac{n^2(m - 1)(m - 3)}{m + 1} - \frac{m^2n^2 - 4mn^2 + m^2n + 3n^2 - 2mn -  2m + 5n - 2}{m  + 1} \\ 
&= -\frac{ m^2n - 2mn - 2m + 5n - 2}{m + 1} = -\frac{(mn - 2)(m-2)  + 5(n-2) + 4}{m+1} < 0.
\end{align*}
Therefore, $LE^+(\mathcal{CCC}(G)) < LE(\mathcal{CCC}(G))$. Thus, if $m = 7$ and $n = 2$
then 
\[
E(\mathcal{CCC}(G)) = LE^+(\mathcal{CCC}(G)) < LE(\mathcal{CCC}(G))
\]
 and  if $m \geq 7$ and $n \geq 3$ or $m \geq 9$ and $n = 2$ then 
\[
E(\mathcal{CCC}(G)) < LE^+(\mathcal{CCC}(G)) < LE(\mathcal{CCC}(G)).
\]  

\noindent \textbf{Case 2.} $m$ is even and $n\geq 2$.

If $m=2$ and $n\geq 2$ then, by Theorem \ref{U(n,m)}, we have
\[
LE^+(\mathcal{CCC}(G)) = E(\mathcal{CCC}(G))= LE(\mathcal{CCC}(G))= 4(n-1).
\]
If $m=4$ and $n\geq 2$ then, by Theorem \ref{U(n,m)}, we have
\[
LE^+(\mathcal{CCC}(G)) = E(\mathcal{CCC}(G))= LE(\mathcal{CCC}(G))= 6(n-1).
\]

If $m=6$ and $n= 2$ then, by Theorem \ref{U(n,m)}, we have
\[
LE^+(\mathcal{CCC}(G)) - E(\mathcal{CCC}(G)) = 2(n + 2)(n - 1) - (n(m + 2) - 6) = -4 < 0
\]
and
\begin{align*}
E(\mathcal{CCC}(G)) &- LE(\mathcal{CCC}(G))\\ &= n(m + 2) - 6 - \frac{2m^2n^2 - 12mn^2 + m^2n + 16n^2 - 4mn -  2m + 12n - 4}{m + 2} = -2 < 0.
\end{align*}
Therefore, $ LE^+(\mathcal{CCC}(G)) < E(\mathcal{CCC}(G)) < LE(\mathcal{CCC}(G)).$

If $m=6$ and $n\geq 3$ then by Theorem \ref{U(n,m)}
\[
E(\mathcal{CCC}(G)) - LE^+(\mathcal{CCC}(G)) = n(m + 2) - 6 - 2(n + 2)(n - 1) 
= 2n( 3 - n ) - 2 < 0
\]
and
\begin{align*}
LE^{+}(\mathcal{CCC}(G)) &- LE(\mathcal{CCC}(G))\\ &= 2(n + 2)(n - 1) - \frac{2m^2n^2 - 12mn^2 + m^2n + 16n^2 - 4mn -  2m + 12n - 4}{m + 2}\\ &= -( n + 2 ) < 0.
\end{align*}
Therefore $ E(\mathcal{CCC}(G)) < LE^+(\mathcal{CCC}(G)) < LE(\mathcal{CCC}(G)).$

If $m \geq 8$ and $n\geq 2$ then, by Theorem \ref{U(n,m)}, we have
\begin{align*}
E(\mathcal{CCC}(G)) - LE^+(\mathcal{CCC}(G))&= n(m + 2) - 6 - \frac{2n^2(m - 2)(m - 4)}{m + 2} \\ 
&= - \frac{2m^2n^2 - 12mn^2 - m^2n + 16n^2 - 4mn +  6m - 4n + 12}{ m + 2 } \\ 
&= - \frac{f_{2}(m, n)}{m+2},
\end{align*}
where $f_{2}(m, n) = mn(2n - 1)(m - 8) + 2m( 2n(n - 3) + 3) + 4n(4n - 1) + 12$.
For $n=2$ and $m\geq 8$ we have $f_{2}(m, n) = (6m - 2)(m-8) + 52 > 0$. For $n\geq 3$ and $m \geq 8$ we have $f_{2}(m, n) > 0$. Therefore, if $m \geq 8$ and $n\geq 2$ then $E(\mathcal{CCC}(G)) < LE^+(\mathcal{CCC}(G))$. 

If $m \geq 8$ and $n\geq 2$ then, by Theorem \ref{U(n,m)}, we also have
\begin{align*}
LE^+(\mathcal{CCC}(G)) &- LE(\mathcal{CCC}(G)) \\
&= \frac{2n^2(m - 2)(m - 4)}{m + 2} - \frac{2m^2n^2 - 12mn^2 + m^2n + 16n^2 - 4mn -  2m + 12n - 4}{m + 2}\\ 
&= - \frac{m^2n - 4mn - 2m + 12n - 4}{m+2}\\
& = - \frac{mn(m - 8) + 2m(2n -1) + 4(3n - 1)}{m + 2} < 0.
\end{align*} 
Therefore, $LE^+(\mathcal{CCC}(G)) < LE(\mathcal{CCC}(G)).$ Thus, if $m \geq 8$ and $n\geq 2$ then 
\[
E(\mathcal{CCC}(G)) < LE^+(\mathcal{CCC}(G)) < LE(\mathcal{CCC}(G)).
\]
Hence, the result follows.
\end{proof}
\begin{theorem}\label{cV(8n)}
If $G = V_{8n}$ then $E(\mathcal{CCC}(G)) \leq LE^+(\mathcal{CCC}(G)) \leq LE(\mathcal{CCC}(G))$. The equality holds if and only if $n = 2$.
\end{theorem}
\begin{proof}
We shall prove the result by considering the following cases.

\noindent \textbf{Case 1.} $n$ is odd.

By Theorem \ref{V(8n)}, we have
\[
E(\mathcal{CCC}(G)) - LE^+(\mathcal{CCC}(G)) = 4n - 4 - \frac{4(2n - 1)(2n - 2)}{2n + 1} = - \frac{4(n - 1)(2n - 3)}{2n + 1}
\]
and
\[
LE^+(\mathcal{CCC}(G)) - LE(\mathcal{CCC}(G)) = \frac{4(2n - 1)(2n - 2)}{2n + 1} - \frac{6(2n - 1)(2n - 2)}{2n + 1} = - \frac{2(2n - 1)(2n - 2)}{2n + 1} < 0.
\]
Therefore, $ E(\mathcal{CCC}(G)) < LE^+(\mathcal{CCC}(G)) < LE(\mathcal{CCC}(G))$.

\noindent \textbf{Case 2.} $n$ is even.

If $n = 2$ then, by Theorem \ref{V(8n)}, we have
\[
E(\mathcal{CCC}(G)) = LE^+(\mathcal{CCC}(G)) = LE(\mathcal{CCC}(G)) = 6.
\]
If $n \geq 4$ then, by Theorem \ref{V(8n)}, we have
\begin{align*}
E(\mathcal{CCC}(G)) - LE^+(\mathcal{CCC}(G)) &= 4n - 2 - \frac{16(n - 1)(n - 2)}{n + 1}\\ &= - \frac{ 2( 6n^2 + 25n - 17 )}{n + 1} = - \frac{ 2( 6n ( n - 4 ) + 49n - 7 )}{n + 1} < 0
\end{align*}
and
\begin{align*}
LE^+(\mathcal{CCC}(G)) - LE(\mathcal{CCC}(G)) &= \frac{16(n - 1)(n - 2)}{n + 1} - \frac{2(2n - 3)(5n - 7)}{n + 1} \\ &= - \frac{2( 2n^2 - 5n + 5 )}{n + 1} = - \frac{2( 2n ( n - 4 ) + 3n + 5 )}{n + 1} < 0.
\end{align*}
Therefore, $ E(\mathcal{CCC}(G)) < LE^+(\mathcal{CCC}(G)) < LE(\mathcal{CCC}(G))$. Hence, the result follows.
\end{proof}
\begin{theorem}\label{cSD(8n)}
If $G = SD_{8n}$ then $E(\mathcal{CCC}(G)) \leq LE^+(\mathcal{CCC}(G)) \leq LE(\mathcal{CCC}(G))$. The equality holds if and only if $n = 3$.
\end{theorem}
\begin{proof}
We shall prove the result by considering the following cases.

\noindent \textbf{Case 1.} $n$ is odd.

If $n=3$ then, by Theorem \ref{SD(8n)}, we have
\[
E(\mathcal{CCC}(G)) = LE^+(\mathcal{CCC}(G)) = LE(\mathcal{CCC}(G)) = 12.
\]
If $n = 5$ then, by Theorem \ref{SD(8n)}, we have
\[
E(\mathcal{CCC}(G)) - LE^+(\mathcal{CCC}(G)) = 4n - 22 = -2 < 0
\]
and
\[
LE^+(\mathcal{CCC}(G)) - LE(\mathcal{CCC}(G)) = 22 - \frac{2(2n - 3)(5n - 11)}{n + 1} = -\frac{32}{3} < 0.
\]
Therefore, $ E(\mathcal{CCC}(G)) < LE^+(\mathcal{CCC}(G)) < LE(\mathcal{CCC}(G))$.

If $n \geq 7$  then, by Theorem \ref{SD(8n)}, we have
\begin{align*}
E(\mathcal{CCC}(G)) - LE^+(\mathcal{CCC}(G)) &= 4n - \frac{16(n - 1)(n - 3)}{n + 1}\\ 
&= -\frac{4(3n^2 - 17n + 12)}{n+1} = -\frac{4(3n(n - 7)+4n +12)}{n+1} < 0
\end{align*}
and
\begin{align*}
LE^+(\mathcal{CCC}(G)) - LE(\mathcal{CCC}(G)) &= \frac{16(n - 1)(n - 3)}{n + 1} - \frac{2(2n - 3)(5n - 11)}{n + 1} \\
&= -\frac{2(2n^2 - 5n + 9)}{ n + 1 } = -\frac{ 2( 2n ( n - 7 ) + 9n + 9 )}{ n + 1 } < 0.
\end{align*}
Therefore, $ E(\mathcal{CCC}(G)) < LE^+(\mathcal{CCC}(G)) < LE(\mathcal{CCC}(G))$.

\noindent \textbf{Case 2.} $n$ is even.

If $n = 2$ then, by Theorem \ref{SD(8n)}, we have
\[
E(\mathcal{CCC}(G)) - LE^+(\mathcal{CCC}(G)) = 4n - 4 - \frac{28}{5} = -\frac{8}{5} < 0
\]
and
\[
LE^+(\mathcal{CCC}(G)) - LE(\mathcal{CCC}(G)) = \frac{28}{5} - \frac{6(2n - 1)(2n - 2)}{2n + 1} = -\frac{8}{5} < 0.
\]
Therefore, $ E(\mathcal{CCC}(G)) < LE^+(\mathcal{CCC}(G)) < LE(\mathcal{CCC}(G))$.

If $n \geq 4$ then, by Theorem \ref{SD(8n)}, we have
\[
E(\mathcal{CCC}(G)) - LE^+(\mathcal{CCC}(G)) = 4n - 4 - \frac{4(2n - 1)(2n - 2)}{2n + 1} = -\frac{4(n - 1)(2n - 3)}{2n + 1} < 0
\]
and
\[
LE^+(\mathcal{CCC}(G)) - LE(\mathcal{CCC}(G)) = \frac{4(2n - 1)(2n - 2)}{2n + 1} - \frac{6(2n - 1)(2n - 2)}{2n + 1} = -\frac{2(2n - 1)(2n - 2)}{2n + 1} < 0.
\]
Therefore, $ E(\mathcal{CCC}(G)) < LE^+(\mathcal{CCC}(G)) < LE(\mathcal{CCC}(G))$. Hence, the result follows.
\end{proof}
Note that Theorems \ref{cD(2n)}--\ref{cSD(8n)} can be summarized   in the following way.
\begin{theorem}\label{summarized}
Let $G$ be a finite non-abelian group. Then we have the following.
\begin{enumerate}
\item If $G$ is isomorphic to $D_{6}, D_{8}, D_{12}, Q_{8}, Q_{12}$, $U_{(n, 2)}, U_{(n, 3)}, U_{(n, 4)} (n\geq 2)$, $V_{16}$ or $SD_{24}$  then 
\[
E(\mathcal{CCC}(G)) = LE^+(\mathcal{CCC}(G)) = LE(\mathcal{CCC}(G)).
\]
\item If $G$ is isomorphic to $D_{20},  Q_{20}$, $ U_{(2, 5)}, U_{(3, 5)}$ or $U_{(2, 6)}$  then 
\[
LE^+(\mathcal{CCC}(G)) < E(\mathcal{CCC}(G)) < LE(\mathcal{CCC}(G)).
\]
\item If $G$ is isomorphic to $D_{14}, D_{16}, D_{18},  D_{2n} (n \geq 11)$, $ Q_{16}, Q_{24}, Q_{4m} (m\geq 8)$, $U_{(n, 5)}, (n \geq 4)$, $U_{(n, m)} (m \geq 6  \text{ and } n \geq 3)$,  $U_{(n, m)} (m \geq 8 \text{ and } n \geq 2)$, $V_{8n} (n \geq 3)$, $SD_{16}$ or  $SD_{8n} (n \geq 4)$    then 
\[
E(\mathcal{CCC}(G)) < LE^+(\mathcal{CCC}(G)) < LE(\mathcal{CCC}(G)).
\]
\item If $G$ is isomorphic to $Q_{28}$ or $U_{(2, 7)}$   then  $E(\mathcal{CCC}(G)) = LE^+(\mathcal{CCC}(G)) < LE(\mathcal{CCC}(G)).$
\item If $G$ is isomorphic to $D_{10}$  then $E(\mathcal{CCC}(G)) < LE^+(\mathcal{CCC}(G)) = LE(\mathcal{CCC}(G))$.
\end{enumerate}
\end{theorem}

We conclude this section with the following remark regarding Conjecture \ref{ai} and Question \ref{ab}. 
\begin{remark}
By Theorem \ref{summarized}, it follows that $E(\mathcal{CCC}(G)) \leq LE(\mathcal{CCC}(G))$ and $LE^+(\mathcal{CCC}(G)) \leq LE(\mathcal{CCC}(G))$ for commuting conjugacy class  graph of the groups  $D_{2n}, Q_{4m}, U_{(n, m)}, V_{8n}$ and $SD_{8n}$. Therefore, Conjecture \ref{ai} holds for commuting conjugacy class  graph of these groups whereas the inequality in Question \ref{ab} does not. However, $LE(\mathcal{CCC}(G)) = LE^+(\mathcal{CCC}(G))$ if $G = D_{6}, D_{8}, D_{10}, D_{12}$,  $Q_{8}, Q_{12}$, $V_{16}$, $SD_{24}$ and $U_{(n,m)}$ where $m = 2,3,4$ and $n\geq 2$.
%
\end{remark}

\section{Hyperenergetic and borderenergetic graphs}
It is well-known that 
\begin{equation}\label{hyper-eq 1}
E(K_n)= LE(K_n)= LE^+(K_n) = 2(n-1).
\end{equation}
A graph ${\mathcal{G}}$ with $n$ vertices is called hyperenergetic, L-hyperenergetic or Q-hyperenergetic according as $E(K_n) < E({\mathcal{G}}), LE(K_n) < LE({\mathcal{G}})$ or $LE^+(K_n) < LE^+({\mathcal{G}})$. Also, ${\mathcal{G}}$ is called borderenergetic, L-borderenergetic and Q-borderenergetic if $E(K_n) = E({\mathcal{G}}), LE(K_n) = LE({\mathcal{G}})$ and $LE^+(K_n) = LE^+({\mathcal{G}})$ respectively. These graphs are considered in \cite{Walikar-99,Gutman-99,Gong-15,Tura-17,Fasfous-20}. In this section we consider commuting conjugacy class graph $\mathcal{CCC}(G)$ for the groups considered in Section 3 and determine whether they are hyperenergetic, L-hyperenergetic or Q-hyperenergetic. We shall also determine whether they are borderenergetic, L-borderenergetic or Q-borderenergetic.


\begin{theorem}
Let $G = D_{2n}$.
\begin{enumerate}
\item If $n$ is odd or $n = 4, 6$ then $\mathcal{CCC}(G)$ is neither hyperenergetic, borderenergetic,  L-hyper-energetic,  L-borderenergetic,  Q-hyperenergetic nor Q-borderenergetic.
\item If $n= 8, 10, 12, 14$ then $\mathcal{CCC}(G)$ is L-hyperenergetic but neither hyperenergetic, borderenergetic, L-borderenergetic,  Q-hyperenergetic nor Q-borderenergetic. 
\item If $n$ is even and $n\geq 16$  then $\mathcal{CCC}(G)$ is L-hyperenergetic and Q-hyperenergetic but neither hyperenergetic,   borderenergetic, L-borderenergetic nor Q-borderenergetic.
\end{enumerate}

\end{theorem}
\begin{proof}
 We shall prove the result by considering the following cases.

\noindent \textbf{Case 1.} $n$ is odd.

 By \cite[Proposition 2.1]{sA2020} we have $\mathcal{CCC}(G) = K_1 \sqcup K_{\frac{n - 1}{2}}$. Therefore, $|V(\mathcal{CCC}(G))|= \frac{n + 1}{2}$.  Using \eqref{hyper-eq 1}, we get
\begin{equation}\label{hyper-D-1}
 E(K_{|V(\mathcal{CCC}(G))|})= LE^+(K_{|V(\mathcal{CCC}(G))|}) = LE(K_{|V(\mathcal{CCC}(G))|}) = n-1
\end{equation}

 If $n=3$ then, by  Theorem \ref{cD(2n)} and   Theorem \ref{CCC(G)-D-2n}, we get 
\begin{equation}\label{hyper-D-2}
 E(\mathcal{CCC}(G)) = LE^+(\mathcal{CCC}(G)) = LE(\mathcal{CCC}(G))= 0 < 2 = E(K_{|V(\mathcal{CCC}(G))|}).
\end{equation}

 If $n=5$ then,   by  Theorem \ref{cD(2n)} and   Theorem \ref{CCC(G)-D-2n}, we get
\begin{equation}\label{hyper-D-3}
 E(\mathcal{CCC}(G)) < LE^+(\mathcal{CCC}(G)) = LE(\mathcal{CCC}(G))=\frac{8}{3}< 4 = E(K_{|V(\mathcal{CCC}(G))|}).
\end{equation}
 Therefore $\mathcal{CCC}(G)$ is neither hyperenergetic nor L-hyperenergetic nor Q-hyperenergetic for $n=5$.

If $n\geq 7$ then,   by  Theorem \ref{cD(2n)} and   Theorem \ref{CCC(G)-D-2n}, we get 
\[
 E(\mathcal{CCC}(G)) < LE^+(\mathcal{CCC}(G)) < LE(\mathcal{CCC}(G))= \frac{(n-3)(n+3)}{n+1}.
\]
Again, 
\[
\frac{(n-3)(n+3)}{n+1} - (n-1) = -\frac{8}{n+1} < 0
\]
Therefore,
\begin{equation}\label{hyper-D-4}
 E(\mathcal{CCC}(G)) < LE^+(\mathcal{CCC}(G)) < LE(\mathcal{CCC}(G)) < n - 1 = E(K_{|V(\mathcal{CCC}(G))|}).
\end{equation}
Hence, in view of \eqref{hyper-D-1}--\eqref{hyper-D-4}, it follows that $\mathcal{CCC}(G)$ is neither hyperenergetic, borderenergetic,  L-hyperenergetic, L-borderenergetic,  Q-hyperenergetic nor Q-borderenergetic.

\noindent \textbf{Case 2.} $n$ is even.

By \cite[Proposition 2.1]{sA2020} we have $\mathcal{CCC}(G) = 2K_1 \sqcup K_{\frac{n}{2}-1}$. Therefore, $|V(\mathcal{CCC}(G))|= \frac{n}{2}+1$. Using \eqref{hyper-eq 1}, we get
\begin{equation}\label{hyper-D-5}
 E(K_{|V(\mathcal{CCC}(G))|})= LE^+(K_{|V(\mathcal{CCC}(G))|}) = LE(K_{|V(\mathcal{CCC}(G))|}) = n.
 \end{equation}

\noindent \textbf{Subcase 2.1.} $\frac{n}{2}$ is even.

If $n= 4$ then,   by  Theorem \ref{cD(2n)} and   Theorem \ref{CCC(G)-D-2n}, we get
\begin{equation}\label{hyper-D-6}
 E(\mathcal{CCC}(G)) = LE^+(\mathcal{CCC}(G)) = LE(\mathcal{CCC}(G))= 0 < 4 =  E(K_{|V(\mathcal{CCC}(G))|}).
\end{equation}
Therefore, by \eqref{hyper-D-5} and by \eqref{hyper-D-6}, it follows that  $\mathcal{CCC}(G)$ is  neither hyperenergetic, borderenergetic,  L-hyperenergetic, L-borderenergetic,  Q-hyperenergetic nor Q-borderenergetic.

 If $n=8$ then,   by  Theorem \ref{cD(2n)} and   Theorem \ref{CCC(G)-D-2n}, we get 
 \[
 E(\mathcal{CCC}(G)) < LE^+(\mathcal{CCC}(G)) = 6 < 8 = E(K_{|V(\mathcal{CCC}(G))|}). 
 \]
Also, 
\[
 LE(\mathcal{CCC}(G)) = 9 > 8 = LE(K_{|V(\mathcal{CCC}(G))|}).
 \]
So, $\mathcal{CCC}(G)$ is L-hyperenergetic but neither hyperenergetic, borderenergetic, L-borderenergetic,  Q-hyperenergetic nor Q-borderenergetic. 

 
If $n\geq 12$ then,   by  Theorem \ref{CCC(G)-D-2n}, we get 
\[
LE(\mathcal{CCC}(G))= \frac{3(n-2)(n-4)}{n+2}.
\]
We have
\[
n - \frac{3(n-2)(n-4)}{n+2} = -\frac{2(n(n-12)+2n+12)}{n+2} < 0.
\]
So,
$
LE(K_{|V(\mathcal{CCC}(G))|}) < LE(\mathcal{CCC}(G))
$ and so $\mathcal{CCC}(G)$ is L-hyperenergetic but not L-borderenergetic.

By  Theorem \ref{cD(2n)} and   Theorem \ref{CCC(G)-D-2n}, we also get 
 \[
 E(\mathcal{CCC}(G)) < LE^+(\mathcal{CCC}(G)) = \frac{2(n-2)(n-4)}{n+2}. 
 \]
We have
\begin{equation}\label{hyper-D-7}
 \frac{2(n-2)(n-4)}{n+2} - n = \frac{n^2 -14n +16}{n+2} = \frac{n(n-16) + 2n +16}{n+2} := f_{1}(n)
\end{equation}
Therefore, for $n=12$, we have $f_{1}(n) < 0$ and so
\[
E(\mathcal{CCC}(G)) < LE^+(\mathcal{CCC}(G)) = \frac{2(n-2)(n-4)}{n+2} < n  = LE^+(K_{|V(\mathcal{CCC}(G))|}).
\]
Thus, if $n = 12$ then $\mathcal{CCC}(G)$ is L-hyperenergetic but neither hyperenergetic, borderenergetic, L-borderenergetic,  Q-hyperenergetic nor Q-borderenergetic.

If $n\geq 16$ then, by \eqref{hyper-D-7}, we have  $f_{1}(n)> 0$ and so $LE^+(\mathcal{CCC}(G)) > n = LE^+(K_{|V(\mathcal{CCC}(G))|})$. Therefore, $\mathcal{CCC}(G)$ is Q-hyperenergetic but not  Q-borderenergetic. Also, $E(\mathcal{CCC}(G)) = n - 4 < n = E(K_{|V(\mathcal{CCC}(G))|})$ and so  $\mathcal{CCC}(G)$ is neither hyperenergetic nor borderenergetic. Thus, if $n\geq 16$ then $\mathcal{CCC}(G)$ is L-hyperenergetic and Q-hyperenergetic but neither hyperenergetic,   borderenergetic, L-borderenergetic nor Q-borderenergetic.

 
\noindent \textbf{Subcase 2.2.}  $\frac{n}{2}$ is odd. 

 
 If $n= 6$ then, by  Theorem \ref{cD(2n)} and   Theorem \ref{CCC(G)-D-2n}, we get 
 \[
 E(\mathcal{CCC}(G)) = LE^+(\mathcal{CCC}(G)) = LE(\mathcal{CCC}(G))= 4 < 6 = E(K_{|V(\mathcal{CCC}(G))|}).
 \]
Therefore, $\mathcal{CCC}(G)$ is neither hyperenergetic, borderenergetic,  L-hyperenergetic, L-borderenergetic,  Q-hyperenergetic nor Q-borderenergetic.

If $n= 10$ then, by  Theorem \ref{cD(2n)} and   Theorem \ref{CCC(G)-D-2n}, we get 
 \[
 LE^+(\mathcal{CCC}(G)) < E(\mathcal{CCC}(G)) < LE(\mathcal{CCC}(G))= 10 = LE(K_{|V(\mathcal{CCC}(G))|}).
 \]
So, $\mathcal{CCC}(G)$ is L-bordererenergetic but neither hyperenergetic, borderenergetic,  L-borderenergetic,  Q-hyperenergetic nor Q-borderenergetic.

If $n\geq 14$ then, by   Theorem \ref{CCC(G)-D-2n}, we get 
\[
LE(\mathcal{CCC}(G))= \frac{(n-4)(3n-10)}{n+2}.
\]
We have
\[
n - \frac{(n-4)(3n-10)}{n+2}  = - \frac{2n(n-14)+4n + 40}{n+2} < 0.
\]
So,
$
LE(K_{|V(\mathcal{CCC}(G))|}) < LE(\mathcal{CCC}(G))
$ and so $\mathcal{CCC}(G)$ is L-hyperenergetic but not L-borderenergetic.

By  Theorem \ref{cD(2n)} and   Theorem \ref{CCC(G)-D-2n}, we also get 
 \[
E(\mathcal{CCC}(G)) < LE^+(\mathcal{CCC}(G)) = \frac{2(n-2)(n-6)}{n+2}. 
 \]
We have
\begin{equation}\label{hyper-D-8}
\frac{2(n-2)(n-6)}{n+2} - n = \frac{n^2-18n+24}{n+2} = \frac{n(n-18)+24}{n+2} := f_{2}(n).
\end{equation}
Therefore, for $n=14$, we have $f_{2}(n) < 0$ and so
\[
E(\mathcal{CCC}(G)) < LE^+(\mathcal{CCC}(G)) = \frac{2(n-2)(n-4)}{n+2} < n  = LE^+(K_{|V(\mathcal{CCC}(G))|}).
\]
Thus, if $n = 14$ then $\mathcal{CCC}(G)$ is L-hyperenergetic but neither hyperenergetic, borderenergetic, L-borderenergetic,  Q-hyperenergetic nor Q-borderenergetic. If $n\geq 18$ then, by \eqref{hyper-D-8}, we have  $f_{2}(n)> 0$ and so $LE^+(\mathcal{CCC}(G)) > n = LE^+(K_{|V(\mathcal{CCC}(G))|})$. Therefore, $\mathcal{CCC}(G)$ is Q-hyperenergetic but not  Q-borderenergetic. Also, $E(\mathcal{CCC}(G)) = n - 2 < n = E(K_{|V(\mathcal{CCC}(G))|})$ and so  $\mathcal{CCC}(G)$ is neither hyperenergetic nor borderenergetic. Thus, if $n\geq 18$ then $\mathcal{CCC}(G)$ is L-hyperenergetic and Q-hyperenergetic but neither hyperenergetic,   borderenergetic, L-borderenergetic nor Q-borderenergetic.
\end{proof}

\begin{theorem}
Let $G= Q_{4m}$.
\begin{enumerate}
\item If $m= 2,3,4$ then $\mathcal{CCC}(G)$ is neither hyperenergetic, borderenergetic,  L-hyperenergetic, L-borderenergetic,  Q-hyperenergetic nor Q-borderenergetic.
\item If $m=5$ then $\mathcal{CCC}(G)$ is L-borderenergetic but neither hyperenergetic, borderenergetic, L-hyperenergetic,  Q-hyperenergetic nor Q-borderenergetic. 
\item If $m=6,7$ then $\mathcal{CCC}(G)$ is L-hyperenergetic but neither hyperenergetic, borderenergetic, L-borderenergetic,  Q-hyperenergetic nor Q-borderenergetic.

\item If $m\geq 8$ then $\mathcal{CCC}(G)$ is L-hyperenergetic and Q-hyperenergetic but neither hyperenergetic,   borderenergetic, L-borderenergetic nor Q-borderenergetic.
\end{enumerate}
\end{theorem}
\begin{proof}
 We shall prove the result by considering the following cases.

\noindent \textbf{Case 1.} $m$ is odd.

 By \cite[Proposition 2.2]{sA2020} we have $\mathcal{CCC}(G) = K_2 \sqcup K_{m - 1}$. Therefore, $|V(\mathcal{CCC}(G))|= m+1$. Using \eqref{hyper-eq 1}, we get
\begin{equation}\label{hyper-Q-1}
 E(K_{|V(\mathcal{CCC}(G))|})= LE^+(K_{|V(\mathcal{CCC}(G))|}) = LE(K_{|V(\mathcal{CCC}(G))|}) = 2m.
\end{equation}

If $m=3$ then, by   Theorem \ref{cQ(4m)} and Theorem \ref{Q(4m)}, we get
\begin{equation}\label{hyper-Q-2}
 E(\mathcal{CCC}(G)) = LE^+(\mathcal{CCC}(G)) = LE(\mathcal{CCC}(G)) = 4 < 6 = E(K_{|V(\mathcal{CCC}(G))|}).
\end{equation}
Therefore, by \eqref{hyper-Q-1} and \eqref{hyper-Q-2},   $\mathcal{CCC}(G)$ is neither hyperenergetic, borderenergetic,  L-hyperenergetic, L-borderenergetic,  Q-hyperenergetic nor Q-borderenergetic.

If $m=5$ then, by   Theorem \ref{cQ(4m)} and Theorem \ref{Q(4m)}, we get
\[
LE^+(\mathcal{CCC}(G)) < E(\mathcal{CCC}(G)) < LE(\mathcal{CCC}(G)) = 10 = LE(K_{|V(\mathcal{CCC}(G))|}).
\]
So, $\mathcal{CCC}(G)$ is L-borderenergetic but neither hyperenergetic, borderenergetic, L-hyperenergetic,  Q-hyperenergetic nor Q-borderenergetic. 

If $m=7$ then, by   Theorem \ref{cQ(4m)} and Theorem \ref{Q(4m)}, we get
\[
LE^+(\mathcal{CCC}(G)) = E(\mathcal{CCC}(G)) = 12 < 14 =  E(K_{|V(\mathcal{CCC}(G))|}). 
\]
Also,
\[
LE(\mathcal{CCC}(G)) = 20 > 14 = LE(K_{|V(\mathcal{CCC}(G))|}).
\]
So,  $\mathcal{CCC}(G)$ is L-hyperenergetic but neither hyperenergetic, borderenergetic, L-borderenergetic,  Q-hyperenergetic nor Q-borderenergetic.

If $m\geq 9$ then, by   Theorem \ref{cQ(4m)} and Theorem \ref{Q(4m)}, we get
\[
\frac{4(m-1)(m-3)}{m+1} = LE^+(\mathcal{CCC}(G)) < LE(\mathcal{CCC}(G)).
\]
We have 
\[
2m - \frac{4(m-1)(m-3)}{m+1}  = -\frac{2(m(m-9)+6)}{m+1} < 0 
\]
and so  $LE^+(K_{|V(\mathcal{CCC}(G))|}) = 2m < \frac{4(m-1)(m-3)}{m+1} = LE^+(\mathcal{CCC}(G)) < LE(\mathcal{CCC}(G)).$
Hence, $\mathcal{CCC}(G)$ is L-hyperenergetic and Q-hyperenergetic but neither  L-borderenergetic nor Q-borderenergetic.
Also, 
\[
E(\mathcal{CCC}(G)) = 2m - 2 <  2m = E(K_{|V(\mathcal{CCC}(G))|}).  
\]
Therefore, $\mathcal{CCC}(G)$ is neither hyperenergetic nor borderenergetic. Thus, if $m\geq 9$ then $\mathcal{CCC}(G)$ is L-hyperenergetic and Q-hyperenergetic but neither hyperenergetic,   borderenergetic, L-borderenergetic nor Q-borderenergetic.

\noindent \textbf{Case 2.} $m$ is even.

By \cite[Proposition 2.2]{sA2020} we have $\mathcal{CCC}(G) = 2K_1 \sqcup K_{m - 1}$. Therefore, $|V(\mathcal{CCC}(G))|= m+1$. Using \eqref{hyper-eq 1}, we get
\begin{equation}\label{hyper-Q-3}
 E(K_{|V(\mathcal{CCC}(G))|})= LE^+(K_{|V(\mathcal{CCC}(G))|}) = LE(K_{|V(\mathcal{CCC}(G))|}) = 2m.
\end{equation}
 If $m=2$ then, by   Theorem \ref{cQ(4m)} and Theorem \ref{Q(4m)}, we get
\begin{equation}\label{hyper-Q-4}
 E(\mathcal{CCC}(G)) = LE^+(\mathcal{CCC}(G)) = LE(\mathcal{CCC}(G)) = 0 < 4 = E(K_{|V(\mathcal{CCC}(G))|}).
\end{equation}
Therefore, by \eqref{hyper-Q-3} and \eqref{hyper-Q-4},   $\mathcal{CCC}(G)$ is neither hyperenergetic, borderenergetic,  L-hyperenergetic, L-borderenergetic,  Q-hyperenergetic nor Q-borderenergetic.
 
  If $m=4$ then, by   Theorem \ref{cQ(4m)} and Theorem \ref{Q(4m)}, we get
\begin{equation}\label{hyper-Q-5}
 E(\mathcal{CCC}(G)) < LE^+(\mathcal{CCC}(G)) < LE(\mathcal{CCC}(G)) = \frac{36}{5} < 8 = E(K_{|V(\mathcal{CCC}(G))|}).
\end{equation}
Therefore, by \eqref{hyper-Q-3} and \eqref{hyper-Q-5},   $\mathcal{CCC}(G)$ is neither hyperenergetic, borderenergetic,  L-hyperenergetic, L-borderenergetic,  Q-hyperenergetic nor Q-borderenergetic.
 
 If $m\geq 6$ then, by Theorem \ref{Q(4m)}, we get
\[
LE(\mathcal{CCC}(G)) = \frac{6(m-1)(m-2)}{m+1}.
\]
We have
\[
2m - \frac{6(m-1)(m-2)}{m+1} = - \frac{4(m^2-5m + 3)}{m+1} = - \frac{4(m(m-6)+ m + 3)}{m+1} < 0
 \]
and so $LE(K_{|V(\mathcal{CCC}(G))|}) = 2m < \frac{6(m-1)(m-2)}{m+1} = LE(\mathcal{CCC}(G)).$ Hence, $\mathcal{CCC}(G)$ is L-hyperenergetic but not L-borderenergetic.

By Theorem \ref{cQ(4m)} and Theorem \ref{Q(4m)}, we also get 
\[
E(\mathcal{CCC}(G)) < LE^+(\mathcal{CCC}(G)) = \frac{4(m-1)(m-2)}{m+1}.
\]
We have 
\begin{equation}\label{hyper-Q-6}
\frac{4(m-1)(m-2)}{m+1} - 2m = \frac{2(m^2-7m+4)}{m+1} = \frac{2(m(m-8)+m+4)}{m+1} = f(m).
\end{equation}
Therefore, for $m = 6$, we have $f(m)< 0$ and so
\[
E(\mathcal{CCC}(G)) < LE^+(\mathcal{CCC}(G)) = \frac{4(m-1)(m-2)}{m+1} < 2m = LE^+(K_{|V(\mathcal{CCC}(G))|}).
\]
Thus, if $m = 6$ then $\mathcal{CCC}(G)$ is L-hyperenergetic but neither hyperenergetic, borderenergetic, L-borderenergetic,  Q-hyperenergetic nor Q-borderenergetic.

If $n\geq 8$ then, by \eqref{hyper-Q-6}, we have  $f(m)> 0$ and so $LE^+(\mathcal{CCC}(G)) > 2m = LE^+(K_{|V(\mathcal{CCC}(G))|})$. Therefore, $\mathcal{CCC}(G)$ is Q-hyperenergetic but not  Q-borderenergetic. Also, $E(\mathcal{CCC}(G)) = 2m - 4 < 2m = E(K_{|V(\mathcal{CCC}(G))|})$ and so  $\mathcal{CCC}(G)$ is neither hyperenergetic nor borderenergetic. Thus, if $n\geq 8$ then $\mathcal{CCC}(G)$ is L-hyperenergetic and Q-hyperenergetic but neither hyperenergetic,   borderenergetic, L-borderenergetic nor Q-borderenergetic.
\end{proof}

\begin{theorem}
Let $G = U_{(n,m)}$.
\begin{enumerate}
\item If $m=2, 3, 4$ and $n\geq 2$ or $m=6$ and $n = 2$ then  $\mathcal{CCC}(G)$ is neither hyperenergetic, borderenergetic, L-hyperenergetic, L-borderenergetic,  Q-hyperenergetic nor Q-borderenergetic. 

\item If  $m=5$ and $n= 2$ then $\mathcal{CCC}(G)$ is L-borderenergetic but neither hyperenergetic, borderenergetic, L-hyperenergetic,   Q-hyperenergetic nor Q-borderenergetic.

\item If $m=5$ and $n=3$, $m = 6$ and $n = 3$ or $m=7$ and $n=2$ then $\mathcal{CCC}(G)$ is L-hyper-energetic but neither hyperenergetic, borderenergetic, L-borderenergetic, Q-hyperenergetic  nor Q-borderenergetic.

\item If  $m = 5, 6$ and $n \geq 4$;  $m=7$ and $n\geq 3$ or $m\geq 8$ and $n \geq 2$ then $\mathcal{CCC}(G)$ is L-hyperenergetic and Q-hyperenergetic but neither hyperenergetic, borderenergetic,  L-border-energetic  nor Q-borderenergetic.
\end{enumerate}
\end{theorem}

\begin{proof}
We shall prove the result by considering the following cases.

\noindent \textbf{Case 1.} $m$ is odd and $n \geq 2$.

 By \cite[Proposition 2.3]{sA2020} we have $\mathcal{CCC}(G) =   K_{\frac{n(m - 1)}{2}} \sqcup K_n$. Therefore, $|V(\mathcal{CCC}(G))|= \frac{n(m + 1)}{2}$. Using \eqref{hyper-eq 1}, we get
\begin{equation}\label{hyper-U-1}
 E(K_{|V(\mathcal{CCC}(G))|})= LE^+(K_{|V(\mathcal{CCC}(G))|}) = LE(K_{|V(\mathcal{CCC}(G))|}) = mn + n - 2.
\end{equation}
By Theorem \ref{U(n,m)} we get
\[
E(\mathcal{CCC}(G)) = mn + n - 4 < mn + n - 2.
\]
Therefore, $\mathcal{CCC}(G)$ is neither hyperenergetic nor borderenergetic.
 
If $m=3$ and $n\geq 2$ then, by  Theorem \ref{U(n,m)}, we get
\[
LE^+(\mathcal{CCC}(G)) = LE(\mathcal{CCC}(G))= 4n - 4 < 4n  - 2 = LE(K_{|V(\mathcal{CCC}(G))|}).
\]
So, $\mathcal{CCC}(G)$ is neither L-hyperenergetic, L-borderenergetic,  Q-hyperenergetic nor Q-borderenergetic.
Thus, if $m=3$ and $n\geq 2$ then  $\mathcal{CCC}(G)$ is neither hyperenergetic, borderenergetic, L-hyperenergetic, L-borderenergetic,  Q-hyperenergetic nor Q-borderenergetic.

  If $m=5$ and $n=2$ then, by  Theorem \ref{cU(n,m)} and Theorem \ref{U(n,m)}, we get
\[
LE^+(\mathcal{CCC}(G)) < LE(\mathcal{CCC}(G))= 10 =  LE(K_{|V(\mathcal{CCC}(G))|}).
\]
Therefore, $\mathcal{CCC}(G)$ is L-borderenergetic but neither L-hyperenergetic,   Q-hyperenergetic nor Q-borderenergetic. Thus, if  $m=5$ and $n= 2$ then $\mathcal{CCC}(G)$ is L-borderenergetic but neither hyperenergetic, borderenergetic, L-hyperenergetic,   Q-hyperenergetic nor Q-borderenergetic.

If $m=5$ and $n=3$ then, by  Theorem \ref{U(n,m)}, we get
\[
LE(\mathcal{CCC}(G))= 20 > 16 = LE(K_{|V(\mathcal{CCC}(G))|}).
\]
Therefore, $\mathcal{CCC}(G)$ is L-hyperenergetic but not L-borderenergetic. Also,
\[
LE^+(\mathcal{CCC}(G)) = 12 < 16 = LE^+(K_{|V(\mathcal{CCC}(G))|}).
\]
Therefore, $\mathcal{CCC}(G)$ is neither Q-hyperenergetic  nor Q-borderenergetic. Thus, if $m=5$ and $n=3$ then $\mathcal{CCC}(G)$ is L-hyperenergetic but neither hyperenergetic, borderenergetic, L-borderenergetic, Q-hyperenergetic  nor Q-borderenergetic.

If $m=5$ and $n\geq 4$ then, by  Theorem \ref{cU(n,m)} and Theorem \ref{U(n,m)}, we get
\[
\frac{2(2n+3)(n-1)}{3} = LE^+(\mathcal{CCC}(G)) < LE(\mathcal{CCC}(G)).
\]
We have
\[
6n - 2 - \frac{2(2n+3)(n-1)}{3} =- \frac{4(n^2 - n - 1)}{3} = -\frac{4(n(n-4)+ 3n - 1)}{3} < 0 
\]
and so $LE^+(K_{|V(\mathcal{CCC}(G))|}) = 6n - 2 < \frac{2(2n+3)(n-1)}{3} = LE^+(\mathcal{CCC}(G)) < LE(\mathcal{CCC}(G)).$
Therefore, $\mathcal{CCC}(G)$ is L-hyperenergetic and Q-hyperenergetic but neither  L-borderenergetic  nor Q-border-energetic. Thus, if $m = 5$ and $n \geq 4$ then $\mathcal{CCC}(G)$ is L-hyperenergetic and Q-hyperenergetic but neither hyperenergetic, borderenergetic,  L-borderenergetic  nor Q-borderenergetic. 

If $m\geq 7$ and $n\geq 2$ then, by Theorem \ref{U(n,m)}, we get
\[
LE(\mathcal{CCC}(G)) = \frac{m^2n^2 - 4mn^2 + m^2n + 3n^2 - 2mn -  2m + 5n - 2}{m  + 1}. 
\]
We have
\begin{align*}
mn + n - 2 - LE(\mathcal{CCC}(G))  
&= -\frac{m^2n^2 - 4mn^2 - 4mn + 3n^2 + 4n}{m+1} \\
&= -\frac{mn^2(m-7) + 2mn(n-2) + mn^2 + 3n^2 + 4n}{m+1} < 0
\end{align*}
and so $LE(K_{|V(\mathcal{CCC}(G))|}) = mn + n - 2 < LE(\mathcal{CCC}(G))$. Therefore, 
$\mathcal{CCC}(G)$ is  L-hyperenergetic but not L-borderenergetic. By Theorem \ref{cU(n,m)} and Theorem \ref{U(n,m)}, we also get
\[
\frac{n^2(m - 1)(m - 3)}{m + 1} = LE^+(\mathcal{CCC}(G)) < LE(\mathcal{CCC}(G)).
\]
Let $f_{1}(m, n) = \frac{n^2(m - 1)(m - 3)}{m + 1} - (mn + n - 2)$. Then
\begin{align*}
f_{1}(m, n)  &= \frac{2 + 2 m - 2mn - m^2n - n + 3 n^2 - 4 m n^2 + m^2 n^2}{m + 1} \\
&= \frac{ mn^2(m-11) + m n^2+ m^2n(n-2) + 2mn(n-2) + 2n(3n - 1) + 2(m + 1)}{2(m + 1)}.
\end{align*}
For $m=7$ and $n=2$ we have $ f_{1}(m, n) = -2 < 0$ and so 
\[
LE^+(\mathcal{CCC}(G)) = \frac{n^2(m - 1)(m - 3)}{m + 1}  <  mn + n - 2 = LE^+(K_{|V(\mathcal{CCC}(G))|}).
\] 
Therefore, $\mathcal{CCC}(G)$ is neither Q-hyperenergetic nor Q-borderenergetic. Thus, if $m=7$ and $n=2$ then $\mathcal{CCC}(G)$ is  L-hyperenergetic but neither hyperenergetic, borderenergetic, L-borderenergetic, Q-hyperenergetic nor Q-borderenergetic.

If $m=7$ and $n\geq 3$ then  $f_{1}(m, n)= \frac{n(3n-8) + 16}{8}> 0$. Therefore, 
\[
LE^+(K_{|V(\mathcal{CCC}(G))|}) = mn + n - 2 < \frac{n^2(m - 1)(m - 3)}{m + 1} = LE^+(\mathcal{CCC}(G))
\]
and so $\mathcal{CCC}(G)$ is Q-hyperenergetic but not Q-borderenergetic. Thus, if 
$m=7$ and $n\geq 3$ then $\mathcal{CCC}(G)$ is L-hyperenergetic and  Q-hyperenergetic but neither hyperenergetic, borderenergetic, L-borderenergetic nor Q-borderenergetic.

Now, for $m=9$ and $n=2$ we have $ f_1(m, n) = \frac{6}{5}> 0$.  For $m=9$ and $n\geq 3$ we have $f_1(m, n)= \frac{2n(12n-25) + 10}{5}> 0$. For $m\geq 11$ and $n\geq 2$ we have $f_1(m, n) > 0$.
Therefore, for $m\geq 9$ and $n \geq 2$ we have
\[
LE^+(K_{|V(\mathcal{CCC}(G))|}) = mn + n - 2 < \frac{n^2(m - 1)(m - 3)}{m + 1} = LE^+(\mathcal{CCC}(G))
\]
and so $\mathcal{CCC}(G)$ is Q-hyperenergetic but not Q-borderenergetic. Thus, if 
$m\geq 9$ and $n \geq 2$ then $\mathcal{CCC}(G)$ is L-hyperenergetic and  Q-hyperenergetic but neither hyperenergetic, borderenergetic, L-borderenergetic nor Q-borderenergetic. 
%

\noindent \textbf{Case 2.} $m$ is even and $n \geq 2$.

 By \cite[Proposition 2.3]{sA2020} we have $\mathcal{CCC}(G) = K_{\frac{n(m - 2)}{2}} \sqcup 2K_n$. Therefore, $|V(\mathcal{CCC}(G))|= \frac{n(m + 2)}{2}$. Using \eqref{hyper-eq 1}, we get
\begin{equation}\label{hyper-U-2}
 E(K_{|V(\mathcal{CCC}(G))|})= LE^+(K_{|V(\mathcal{CCC}(G))|}) = LE(K_{|V(\mathcal{CCC}(G))|}) = mn + 2n - 2.
\end{equation}
By Theorem \ref{U(n,m)} we get
\[
E(\mathcal{CCC}(G)) = 4n - 4 < 4n  - 2 = E(K_{|V(\mathcal{CCC}(G))|}),
\]
if $m = 2$. If $m \geq 4$ then
\[
E(\mathcal{CCC}(G)) = mn + 2n - 6 < mn + 2n - 2 = E(K_{|V(\mathcal{CCC}(G))|}).
\]
Therefore, $\mathcal{CCC}(G)$ is neither hyperenergetic nor borderenergetic.

If $m = 2$ and $n\geq 2$ then, by  Theorem \ref{U(n,m)}, we get
\[
LE^+(\mathcal{CCC}(G)) = LE(\mathcal{CCC}(G)) = 4n-4< 4n-2.
\]
So, $\mathcal{CCC}(G)$ is neither L-hyperenergetic, L-borderenergetic, Q-hyperenergetic nor Q-borderenergetic. Thus, if  $m=2$ and $n\leq 2$ then $\mathcal{CCC}(G)$ is neither hyperenergetic,  borderenergetic, L-hyperenergetic, L-borderenergetic, Q-hyperenergetic nor Q-borderenergetic.

If $m = 4$ and $n \geq 2$ then, by  Theorem \ref{U(n,m)}, we get
\[
LE^+(\mathcal{CCC}(G)) = LE(\mathcal{CCC}(G)) = 6n - 6< 6n - 2.
\]
So, $\mathcal{CCC}(G)$ is neither L-hyperenergetic, L-borderenergetic, Q-hyperenergetic nor Q-borderenergetic. Thus, if  $m=4$ and $n\leq 2$ then $\mathcal{CCC}(G)$ is neither hyperenergetic,  borderenergetic, L-hyperenergetic, L-borderenergetic, Q-hyperenergetic nor Q-borderenergetic.

If $m = 6$ and $n = 2$ then, Theorem \ref{cU(n,m)} and Theorem \ref{U(n,m)}, we get
\[
LE^+(\mathcal{CCC}(G)) <  LE(\mathcal{CCC}(G)) = 12 < 14 = LE(K_{|V(\mathcal{CCC}(G))|}).
\]
So, $\mathcal{CCC}(G)$ is neither L-hyperenergetic, L-borderenergetic, Q-hyperenergetic nor Q-borderenergetic. Thus, if  $m=6$ and $n = 2$ then $\mathcal{CCC}(G)$ is neither hyperenergetic,  borderenergetic, L-hyperenergetic, L-borderenergetic, Q-hyperenergetic nor Q-borderenergetic.

If $m = 6$ and $n \geq 3$ then, by  Theorem \ref{U(n,m)}, we get
\[
LE(\mathcal{CCC}(G))  = 2n^2 + 3n - 2.
\]
We have
\[
8n - 2 - (2n^2 + 3n - 2)   = -n(2n - 5) < 0.
\]
Therefore, $LE(K_{|V(\mathcal{CCC}(G))|}) = 8n - 2 < 2n^2 + 3n - 2 < LE(\mathcal{CCC}(G))$ and so $\mathcal{CCC}(G)$ is L-hyperenergetic but not L-borderenergetic.
By  Theorem \ref{U(n,m)}, we also get
\[
LE^+(\mathcal{CCC}(G)) = 2(n + 2)(n - 1).
\]
Let $g(n) = 2(n + 2)(n - 1) - (8n - 2)$. Then $g(n) = 2(n(n-4) + n - 1)$. Therefore, if $n=3$ then $g(n)= -2 < 0$ and so
\[
LE^+(\mathcal{CCC}(G)) = 2(n + 2)(n - 1) < 8n - 2 = LE^+(K_{|V(\mathcal{CCC}(G))|}).
\]
Therefore, $\mathcal{CCC}(G)$ is neither Q-hyperenergetic nor  Q-borderenergetic. Thus, if $m = 6$ and $n = 3$ then  $\mathcal{CCC}(G)$ is L-hyperenergetic but neither hyperenergetic, borderenergetic, L-borderenergetic, Q-hyperenergetic nor  Q-borderenergetic.

If $n\geq 4$ then    $g(n) > 0$ and so
\[
LE^+(K_{|V(\mathcal{CCC}(G))|}) = 8n - 2 < 2(n + 2)(n - 1) = LE^+(\mathcal{CCC}(G)).
\]
Therefore, $\mathcal{CCC}(G)$ is Q-hyperenergetic but not Q-borderenergetic. Thus, if $m = 6$ and $n\geq 4$ then $\mathcal{CCC}(G)$ is L-hyperenergetic and Q-hyperenergetic but neither  hyperenergetic, borderenergetic, L-borderenergetic nor  Q-borderenergetic.




If $m \geq 8$ and $n \geq 2$ then, by Theorem \ref{cU(n,m)} and   Theorem \ref{U(n,m)}, we get
\[
\frac{2n^2(m - 2)(m - 4)}{m + 2}= LE^+(\mathcal{CCC}(G)) < LE(\mathcal{CCC}(G)).
\]
We have
\begin{align*}
mn + 2n - 2 - \frac{2n^2(m - 2)(m - 4)}{m + 2} &= -\frac{4 + 2 m  -  m^2n - 4 n - 4 m n + 16 n^2 - 12 m n^2 +  2 m^2 n^2}{m + 2}\\
& = -f_{2}(m, n),
\end{align*}
where $f_{2}(m, n) = \frac{mn^2(m - 12) + m^2n(n - 2) + mn(m - 6) + 2n(m - 2) + 16n^2 + 2m + 4}{m + 2}$.

For $m=8$ and $n=2$ we have $f_{2}(m,n) = \frac{6}{5}> 0$.  For $m=8$ and $n\geq 3$ we have $f_{2}(m,n) = \frac{2}{5}(12n^2 - 25n + 5) = \frac{2}{5}(12n(n-3) + 11n + 5) > 0$. For $m=10$ and $n\geq 2$ we have $f_{2}(m, n)= 2(4n^2 - 6n + 1)= 2(4n(n-2) + 2n + 1)> 0$. For $m\geq 12$ and $n\geq 2$ we have $f_{2}(m, n) > 0$. Therefore,
\[
LE^+(K_{|V(\mathcal{CCC}(G))|}) = mn + 2n - 2 < \frac{2n^2(m - 2)(m - 4)}{m + 2} = LE^+(\mathcal{CCC}(G)) < LE(\mathcal{CCC}(G)) 
\]
and so $\mathcal{CCC}(G)$ is L-hyperenergetic and Q-hyperenergetic but neither  L-borderenergetic nor  Q-borderenergetic. Thus, if  $m\geq 8$ and $n \geq 2$ then 
$\mathcal{CCC}(G)$ is L-hyperenergetic and Q-hyperenergetic but neither  hyperenergetic, borderenergetic, L-borderenergetic nor  Q-borderenergetic.
\end{proof}

\begin{theorem}
Let $G= V_{8n}$.
\begin{enumerate}
\item If $n= 2$ then $\mathcal{CCC}(G)$ is neither hyperenergetic, borderenergetic,  L-hyperenergetic, L-border-energetic,  Q-hyperenergetic nor Q-borderenergetic.
\item If $n= 3, 4$ then $\mathcal{CCC}(G)$ is L-hyperenergetic but neither hyperenergetic, borderenergetic, L-borderenergetic,  Q-hyperenergetic nor Q-borderenergetic.
\item If $n \geq 5$ then $\mathcal{CCC}(G)$ is L-hyperenergetic and Q-hyperenergetic but neither hyperenergetic,   borderenergetic, L-borderenergetic nor Q-borderenergetic.
\end{enumerate}
\end{theorem}
\begin{proof}
We shall prove the result by considering the following cases.

\noindent \textbf{Case 1.} $n$ is odd.

 By \cite[Proposition 2.4]{sA2020} we have $\mathcal{CCC}(G) = 2K_1 \sqcup K_{2n - 1}$.  Therefore, $|V(\mathcal{CCC}(G))|= 2n + 1$. Using \eqref{hyper-eq 1}, we get
\begin{equation}\label{hyper-V-1}
 E(K_{|V(\mathcal{CCC}(G))|})= LE^+(K_{|V(\mathcal{CCC}(G))|}) = LE(K_{|V(\mathcal{CCC}(G))|}) = 4n.
 \end{equation}
By  Theorem \ref{V(8n)} we get
\[
LE(\mathcal{CCC}(G))=\frac{6(2n-1)(2n-2)}{2n+1}.
\]
We have
\[ 
4n - \frac{6(2n-1)(2n-2)}{2n+1} = -\frac{4(4n^2 - 10n + 3)}{2n+1} = -\frac{4(4n(n-3)+2n+3)}{2n+1} < 0
 \]
and so $LE(K_{|V(\mathcal{CCC}(G))|}) = 4n < \frac{6(2n-1)(2n-2)}{2n+1} = LE(\mathcal{CCC}(G))$. Hence, $\mathcal{CCC}(G)$ is L-hyperenergetic but not L-borderenergetic.

By  Theorem \ref{cV(8n)} and Theorem \ref{V(8n)}, we also get
\[
 E(\mathcal{CCC}(G)) < LE^+(\mathcal{CCC}(G))  = \frac{4(2n-1)(2n-2)}{2n+1}.
\]
We have
\begin{equation}\label{hyper-V-2}
\frac{4(2n-1)(2n-2)}{2n+1} - 4n = \frac{4(2n^2 - 7n + 2)}{2n+1} = \frac{4(2n(n-5) +3n + 2)}{2n+1} := g_1(n).
\end{equation}
Therefore, for $n = 3$, we have $g_1(n) < 0$ and so
\[
 E(\mathcal{CCC}(G)) < LE^+(\mathcal{CCC}(G))  = \frac{4(2n-1)(2n-2)}{2n+1} < 4n = LE^+(K_{|V(\mathcal{CCC}(G))|}).
\]
Thus, if $n = 3$ then $\mathcal{CCC}(G)$ is L-hyperenergetic but neither hyperenergetic, borderenergetic, L-borderenergetic,  Q-hyperenergetic nor Q-borderenergetic. If  $n\geq 5$ then, by \eqref{hyper-V-2},  we have $g_1(n) > 0$ and so $LE^+(\mathcal{CCC}(G)) > 4n = LE^+(K_{|V(\mathcal{CCC}(G))|})$. Therefore, $\mathcal{CCC}(G)$ is Q-hyperenergetic but not  Q-borderenergetic. Also, $E(\mathcal{CCC}(G)) = 4n - 4 < 4n = E(K_{|V(\mathcal{CCC}(G))|})$ and so  $\mathcal{CCC}(G)$ is neither hyperenergetic nor borderenergetic. Thus, if $n\geq 5$ then $\mathcal{CCC}(G)$ is L-hyperenergetic and Q-hyperenergetic but neither hyperenergetic,   borderenergetic, L-borderenergetic nor Q-borderenergetic.


\noindent \textbf{Case 2.} $n$ is even.

 By \cite[Proposition 2.4]{sA2020} we have $\mathcal{CCC}(G) = 2K_2 \sqcup K_{2n - 2}$. Therefore $|V(\mathcal{CCC}(G))|= 2n + 2$. Using \eqref{hyper-eq 1}, we get
\begin{equation}\label{hyper-V-3}
 E(K_{|V(\mathcal{CCC}(G))|})= LE^+(K_{|V(\mathcal{CCC}(G))|}) = LE(K_{|V(\mathcal{CCC}(G))|}) = 4n + 2.
\end{equation}
If $n= 2$ then, by   Theorem \ref{cV(8n)} and Theorem \ref{V(8n)}, we get 
\begin{equation}\label{hyper-V-4}
 E(\mathcal{CCC}(G)) = LE^+(\mathcal{CCC}(G)) = LE(\mathcal{CCC}(G))= 6 < 10 = E(K_{|V(\mathcal{CCC}(G))|}).
\end{equation}
 Therefore, by \eqref{hyper-V-3} and \eqref{hyper-V-4}, we have   $\mathcal{CCC}(G)$ is neither hyperenergetic, borderenergetic,  L-hyperenergetic, L-borderenergetic,  Q-hyperenergetic nor Q-borderenergetic.

If $n\geq 4$ then,  Theorem \ref{V(8n)}, we get 
\[
  E(\mathcal{CCC}(G)) = 4n - 2 < 4n + 2 = E(K_{|V(\mathcal{CCC}(G))|}).
\]
Therefore, $\mathcal{CCC}(G)$ is neither hyperenergetic nor borderenergetic.

By  Theorem \ref{cV(8n)} and  Theorem \ref{V(8n)}, we also get 
\[
\frac{16(n - 1)(n - 2)}{n + 1} = LE^+(\mathcal{CCC}(G)) < LE(\mathcal{CCC}(G)).
\] 
We have
\begin{equation}\label{hyper-V-5}
 \frac{16(n - 1)(n - 2)}{n + 1} - (4n + 2) = \frac{6(2n^2 - 9n + 5)}{n+1} = \frac{6(2n(n-6)+ 3n + 5)}{n+1} := g_{2}(n).
\end{equation}
Therefore, for $n =4$ we have $g_{2}(n) < 0$ and so $LE^+(\mathcal{CCC}(G)) = \frac{16(n - 1)(n - 2)}{n + 1} < 4n + 2 = LE^+(K_{|V(\mathcal{CCC}(G))|})$. Therefore, $\mathcal{CCC}(G)$ is neither Q-hyperenergetic nor Q-borderenergetic. Also,
\[
LE(\mathcal{CCC}(G)) = \frac{130}{5} = 26 >  18 = LE(K_{|V(\mathcal{CCC}(G))|}).
\]
Therefore, $\mathcal{CCC}(G)$ is   L-hyperenergetic but not L-borderenergetic. Thus, if $n = 4$ then $\mathcal{CCC}(G)$ is   L-hyperenergetic but neither hyperenergetic, borderenergetic,  L-borderenergetic, Q-hyperenergetic nor Q-borderenergetic.

If $n\geq 6$ then, by \eqref{hyper-V-5}, we have $g_2(n) > 0$ and so
 \[
LE^+(K_{|V(\mathcal{CCC}(G))|}) = 4n + 2 < \frac{16(n - 1)(n - 2)}{n + 1} = LE^+(\mathcal{CCC}(G)) < LE(\mathcal{CCC}(G)).
\]
Therefore, $\mathcal{CCC}(G)$ is  L-hyperenergetic and Q-hyperenergetic but neither L-borderenergetic nor Q-borderenergetic. Thus, if $n\geq 6$ then $\mathcal{CCC}(G)$ is  L-hyperenergetic and Q-hyperenergetic but neither hyperenergetic,  borderenergetic, L-borderenergetic nor Q-borderenergetic.  
\end{proof}
\begin{theorem}
Let $G = SD_{8n}$.
\begin{enumerate}
\item If $n= 2, 3$ then  $\mathcal{CCC}(G)$ is neither hyperenergetic, borderenergetic, L-hyperenergetic, L-borderenergetic,  Q-hyperenergetic nor Q-borderenergetic.
\item If $n=5$ then $\mathcal{CCC}(G)$ is L-hyperenergetic and Q-borderenergetic but neither hyperenergetic, borderenergetic,  L-borderenergetic nor Q-hyperenergetic.
\item If $n = 4$ or $n\geq 6$ then $\mathcal{CCC}(G)$ is L-hyperenergetic and Q-hyperenergetic but neither hyperenergetic, borderenergetic,  L-borderenergetic nor  Q-borderenergetic.
\end{enumerate}
\end{theorem}
\begin{proof}
We shall prove the result by considering the following cases.

\noindent \textbf{Case 1.} $n$ is odd.

 By \cite[Proposition 2.5]{sA2020} we have $\mathcal{CCC}(G) = K_4 \sqcup K_{2n - 2}$. Therefore, $|V(\mathcal{CCC}(G))|= 2n + 2$. Using \eqref{hyper-eq 1}, we get
\begin{equation}\label{hyper-SD-1}
 E(K_{|V(\mathcal{CCC}(G))|})= LE^+(K_{|V(\mathcal{CCC}(G))|}) = LE(K_{|V(\mathcal{CCC}(G))|}) = 4n + 2.
\end{equation}
By Theorem \ref{SD(8n)} we get
\[
E(\mathcal{CCC}(G)) = 4n < 4n + 2.
\]
Therefore, $\mathcal{CCC}(G)$ is neither hyperenergetic nor borderenergetic.

If $n=3$ then, by  Theorem \ref{SD(8n)}, we get
\[
LE^+(\mathcal{CCC}(G)) = LE(\mathcal{CCC}(G))= 12 < 14 = LE(K_{|V(\mathcal{CCC}(G))|}).
\] 
So, $\mathcal{CCC}(G)$ is neither L-hyperenergetic, L-borderenergetic,  Q-hyperenergetic nor Q-borderenergetic.
Thus, if $n=3$ then  $\mathcal{CCC}(G)$ is neither hyperenergetic, borderenergetic, L-hyperenergetic, L-border-energetic,  Q-hyperenergetic nor Q-borderenergetic.
 
 If $n=5$ then, by Theorem \ref{cSD(8n)} and Theorem \ref{SD(8n)}, we get
\[
LE^+(K_{|V(\mathcal{CCC}(G))|}) = 22 = LE^+(\mathcal{CCC}(G)) < LE(\mathcal{CCC}(G)).
\] 
Therefore, $\mathcal{CCC}(G)$ is L-hyperenergetic and Q-borderenergetic but neither  L-borderenergetic nor Q-hyperenergetic. Thus, if $n=5$ then $\mathcal{CCC}(G)$ is L-hyperenergetic and Q-borderenergetic but neither hyperenergetic, borderenergetic,  L-borderenergetic nor Q-hyperenergetic.
 
 If $n\geq 7$ then, by   Theorem \ref{cSD(8n)} and Theorem \ref{SD(8n)}, we get
\[
\frac{16(n-1)(n-3)}{n+1} = LE^+(\mathcal{CCC}(G)) < LE(\mathcal{CCC}(G)).
\]
We have 
\[
4n+2 - \frac{16(n-1)(n-3)}{n+1}  = -\frac{2(6n^2-35n + 23)}{n+1} = -\frac{2(6n(n-7)+7n + 23)}{n+1} < 0.
\]
So, $LE^+(K_{|V(\mathcal{CCC}(G))|}) = 4n+2 < \frac{16(n-1)(n-3)}{n+1} = LE^+(\mathcal{CCC}(G)) < LE(\mathcal{CCC}(G))$ and so
$\mathcal{CCC}(G)$ is L-hyperenergetic and Q-hyperenergetic but neither  L-borderenergetic nor  Q-borderenergetic. Thus, if   $n\geq 7$ then $\mathcal{CCC}(G)$ is L-hyperenergetic and Q-hyperenergetic but neither hyperenergetic, borderenergetic,  L-borderenergetic nor  Q-borderenergetic.

\noindent \textbf{Case 2.} $n$ is even.

 By \cite[Proposition 2.5]{sA2020} we have $\mathcal{CCC}(G) = 2K_1 \sqcup K_{2n - 1}$. Therefore,  $|V(\mathcal{CCC}(G))|= 2n + 1$.  Using \eqref{hyper-eq 1}, we get
\begin{equation}\label{hyper-SD-2}
 E(K_{|V(\mathcal{CCC}(G))|})= LE^+(K_{|V(\mathcal{CCC}(G))|}) = LE(K_{|V(\mathcal{CCC}(G))|}) = 4n.
\end{equation}
By Theorem \ref{SD(8n)} we get
\[
E(\mathcal{CCC}(G)) = 4n- 4 < 4n.
\]
Therefore, $\mathcal{CCC}(G)$ is neither hyperenergetic nor borderenergetic.
 
 If $n=2$ then, by  Theorem \ref{cSD(8n)} and Theorem \ref{SD(8n)}, we get
\[
LE^+(\mathcal{CCC}(G)) < LE(\mathcal{CCC}(G))= \frac{36}{5} < 8 = LE(K_{|V(\mathcal{CCC}(G))|}).
\] 
So, $\mathcal{CCC}(G)$ is neither L-hyperenergetic, L-borderenergetic,  Q-hyperenergetic nor Q-borderenergetic. Thus, if $n=2$ then  $\mathcal{CCC}(G)$ is neither hyperenergetic, borderenergetic, L-hyperenergetic, L-border-energetic,  Q-hyperenergetic nor Q-borderenergetic.

If $n\geq 4$ then, by  Theorem \ref{cSD(8n)} and Theorem \ref{SD(8n)}, we get
\[
\frac{4(2n-1)(2n-2)}{2n+1} = LE^+(\mathcal{CCC}(G)) < LE(\mathcal{CCC}(G)).
\] 
We have
\[
4n - \frac{4(2n-1)(2n-2)}{2n+1}  = - \frac{4(2n^2 - 7n + 2)}{2n+1} = - \frac{4(2n(n-4)+ n + 2)}{2n+1} < 0.
\]
Therefore, $LE^+(K_{|V(\mathcal{CCC}(G))|}) = 4n < \frac{4(2n-1)(2n-2)}{2n+1} = LE^+(\mathcal{CCC}(G)) < LE(\mathcal{CCC}(G))$ and so 
$\mathcal{CCC}(G)$ is L-hyperenergetic and Q-hyperenergetic but neither  L-borderenergetic nor  Q-borderener-getic. Thus, if   $n\geq 4$ then $\mathcal{CCC}(G)$ is L-hyperenergetic and Q-hyperenergetic but neither hyperenergetic, borderenergetic,  L-borderenergetic nor  Q-borderenergetic.
\end{proof}

We conclude this paper with the following characterization of commuting conjugacy class graph.
\begin{theorem}
Let $G$ be a finite non-abelian group. Then
\begin{enumerate}
\item $\mathcal{CCC}(G)$ is neither hyperenergetic, borderenergetic,  L-hyperenergetic,  L-borderenergetic,  Q-hyperenergetic nor Q-borderenergetic if $G$ is isomorphic to  $D_{8}, \, D_{12}, D_{2n}(\text{$n$ is odd}),\, Q_{8}, Q_{12}$, $Q_{16}, U_{(2,6)}, \, U_{(n,2)}$,  $U_{(n,3)}, U_{(n,4)}  (n\geq 2),  V_{16}, SD_{16}$ or $SD_{24}$.

\item $\mathcal{CCC}(G)$ is L-borderenergetic but neither hyperenergetic, borderenergetic, L-borderenergetic,  Q-hyperenergetic nor Q-borderenergetic if $G$ is isomorphic to $Q_{20}$ or $U_{(2,5)}$.

\item $\mathcal{CCC}(G)$ is L-hyperenergetic but neither hyperenergetic, borderenergetic, L-borderenergetic,  Q-hyperenergetic nor Q-borderenergetic if $G$ is isomorphic to $D_{16}, D_{20}, D_{24}, D_{28}, Q_{24}, Q_{28}$, $U_{(3,5)}$, $U_{(3,6)}, U_{(2,7)}$, $V_{24}$ or $V_{32}$.

\item $\mathcal{CCC}(G)$ is L-hyperenergetic and Q-borderenergetic but neither hyperenergetic, borderenergetic,  L-borderenergetic nor Q-hyperenergetic if $G$ is isomorphic to $SD_{40}$.

\item $\mathcal{CCC}(G)$ is L-hyperenergetic and Q-hyperenergetic but neither hyperenergetic,   borderenergetic, L-borderenergetic nor Q-borderenergetic if $G$ is isomorphic to $D_{2n}(\text{$n$ is even}, n\geq 16), Q_{4m}(m\geq 8), U_{(n,5)}(n\geq 4), U_{(n,6)}(n\geq 4), U_{(n,7)}(n\geq 3), U_{(n,m)}(n\geq 2 \text{ and } m\geq 8), V_{8n}$ $(n\geq 5), SD_{32}$ or $SD_{8n}(n\geq 6)$. 
\end{enumerate}
\end{theorem}

\begin{theorem}
Let $G$ be a finite non-abelian group. Then
\begin{enumerate}
\item If $G$ is isomorphic to $D_{2n}, Q_{4m}, U_{(n, m)}, V_{8n}$ or $SD_{8n}$  then $\mathcal{CCC}(G)$ is neither  hyperenergetic nor borderenergetic.

\item If $G$ is isomorphic to $D_{2n}(\text{$n$ is even}, n\geq 8),  Q_{4m}(m\geq 6)$, $U_{(n,5)} (n\geq 3)$, $U_{(n,6)} (n\geq 3), U_{(n,m)}(n\geq 2 \text{ and } m\geq 7)$, $V_{8n}(n\geq 3)$ or $SD_{8n}(n\geq 4)$ 
then $\mathcal{CCC}(G)$ is L-hyperenergetic
\item If $G$ is isomorphic to $Q_{20}$ or $U_{(2,5)}$ then $\mathcal{CCC}(G)$ is L-borderenergetic.

\item If $G$ is isomorphic to $D_{2n}(\text{$n$ is even}, n\geq 16), Q_{4m}(m\geq 8), U_{(n,5)}(n\geq 4), U_{(n,6)}(n\geq 4), U_{(n,7)}$ $(n\geq 3), U_{(n,m)}(n\geq 2 \text{ and } m\geq 8), V_{8n}(n\geq 5), SD_{32}$ or $SD_{8n}(n\geq 6)$
 then $\mathcal{CCC}(G)$ is Q-hyperenergetic.

\item If $G$ is isomorphic to  $SD_{40}$ then $\mathcal{CCC}(G)$ is Q-borderenergetic. 
\end{enumerate}
\end{theorem}
\noindent {\bf Acknowledgements}
    The first author is thankful to Council of Scientific and Industrial Research  for the fellowship (File No. 09/796(0094)/2019-EMR-I).


\begin{thebibliography}{30}


\bibitem{Abreu08}
N. M. M. Abreu,  C. T. M. Vinagre, A. S. Bonif$\acute{\rm a}$cioa and I. Gutman, The Laplacian energy of some Laplacian integral graph, \emph{MATCH Commun. Math. Comput. Chem.}, {\bf 60},  447--460 (2008)

















\bibitem{bCrS03} 
K. Bali$\acute{\rm n}$ska, D. Cvetkovi$\acute{\rm c}$, Z. Radosavljevi$\acute{\rm c}$, S. Simi$\acute{\rm c}$ and D. Stevanovi$\acute{\rm c}$, A survey on integral graphs, {\em Univ. Beograd. Publ. Elektrotehn. Fak. Ser. Mat.}, {\bf 13}, 42--65 (2003).

\bibitem{bF1955}
R. Brauer and K. A. Fowler, On groups of even order, {\em Ann. of Math.}   {\bf 62}(2)  565--583 (1955).






\bibitem{Simic07}
 D. Cvetkovi$\acute{\rm c}$, P. Rowlinson, S. Simi$\acute{\rm c}$, Signless Laplacian of finite graphs, \emph{Linear Algebra Appl.}, {\bf 423},  155--171 (2007).




\bibitem{dBN2020}
 P. Dutta, B. Bagchi and R. K. Nath, Various energies of commuting graphs of finite nonabelian groups, \emph{ Khayyam J. Math.}, {\bf 6}(1), 27--45 (2020).


\bibitem{Dutta16}
J. Dutta and R. K. Nath, Spectrum of commuting  graphs of some classes of finite groups, {\em Matematika}, {\bf 33}(1), 87--95 (2017).

\bibitem{DN16} 
J. Dutta and R. K. Nath,   Finite groups whose commuting graphs are integral, {\em Matematicki Vesnik}, {\bf 69}(3), 226--230 (2017). 

\bibitem{dN2018}
J. Dutta and R. K. Nath, Laplacian and signless Laplacian spectrum of commuting graphs of finite groups, {\em Khayyam J.  Math.}, {\bf 4}(1),  77--87 (2018). 




\bibitem{Fasfous-20}
W. N. T. Fasfous, R. K. Nath and R. Sharafdini, Various spectra and energies of commuting graphs of finite rings, {\it Hacettepe Journal of Mathematics and Statistics}, accepted for publication.



\bibitem{Gong-15}
S.C. Gong, X. Li, G.H. Xu, I. Gutman and B. Furtula, 
{\it Borderenergetic graphs}, 
 MATCH Commun. Math. Comput. Chem.
{\bf 74}, 321-332, 2015.

\bibitem{Gutman-99}
I. Gutman,
{\it Hyperenergetic molecular graphs},
J. Serb. Chem. Soc.
{\bf 64}, 199-205, 1999.

\bibitem{gavbr}
I. Gutman,  N. M. M. Abreu, , C. T. M. Vinagre,  A. S. Bonif\'{a}cioa  and S. Radenkovi\'{c},   Relation between energy and Laplacian energy. {\em  MATCH Communications in Mathematical and in Computer Chemistry}, {\bf 59}, 343--354, (2008).




\bibitem{hS74}
F. Harary and A. J. Schwenk, Which graphs have integral spectra?,
\textit{Graphs and Combin.}, Lect. Notes Math., Vol 406, Springer-Verlag,
Berlin,  45--51 (1974).

\bibitem{hLM2009} 
M. Herzog, M. Longobardi  and M. Maj, On a commuting graph on conjugacy classes of groups. {\em  Communications in Algebra}, {\bf 37}(10), 3369--3387, (2009).






\bibitem{Kirkland07}
 S. Kirkland, Constructably Laplacian integral graphs, \emph{Linear Algebra  Appl.}, {\bf 423}, 3--21 (2007). 
 
 
\bibitem{ll}
J. Liu,  and B. Liu,  On the relation between energy and Laplacian energy. {\em  MATCH Communications in Mathematical and in Computer Chemistry}, {\bf 61}, 403--406 (2009).

 
 

\bibitem{Merries94} 
R. Merris, Degree maximal graphs are Laplacian integral, \emph{Linear Algebra Appl.},
{\bf 199},  381--389 (1994). 

\bibitem{mefw}
A. Mohammadian,  A. Erfanian, D. G. M. Farrokhi  and  B. Wilkens,  Triangle-free commuting conjugacy class graphs. {\em  Journal of Group Theory}, {\bf 19}, 1049--1061, (2016).


 

 

\bibitem{nath17} 
R. K. Nath, Various spectra of commuting graphs $n$-centralizer finite groups, International Journal of Engineering Science and Technology, accepted for publication.

 
 

\bibitem{sA2020}
M. A. Salahshour and A. R. Ashrafi, Commuting conjugacy class graph of finite CA-groups, 
{\em Khayyam J. Math.}, {\bf 6}(1), 108--118. 

\bibitem{Simic08}
S. K. Simi$\acute{\rm c}$ and   Z. Stani$\acute{\rm c}$,  $Q$-integral graphs with edge-degrees at most five, \emph{Discrete Math.}, {\bf 308},  4625--4634 (2008).





\bibitem{ssm}
D. Stevanovi\'{c}, I. Stankovi\'{c},  and M. Milo\v{s}evi\'{c},  More on the relation between energy and Laplacian energy. {\em  MATCH Communications in Mathematical and in Computer Chemistry}, {\bf 61}, 395--401 (2008).

\bibitem{Tura-17}
F. Tura, 
{\it L-borderenergetic graphs}, 
  MATCH Commun. Math. Comput. Chem.
{\bf 77}, 37-44, 2017.





\bibitem{Walikar-99}
H.B. Walikar, H.S. Ramane and P.R. Hampiholi, 
{\it On the energy of a graph},
 Graph Connections, Allied Publishers, New Delhi, 1999, pp. 120-123, Eds.  R. Balakrishnan,
H. M. Mulder, A. Vijayakumar.

\end{thebibliography}
\end{document}